\documentclass[twoside,a4paper]{article}
\usepackage[english]{babel}
\usepackage[latin1]{inputenc}
\usepackage{amsmath,amssymb,amsthm}
\usepackage{amsfonts}
\usepackage{latexsym}
\usepackage{graphics, graphicx}
\usepackage{upref}
\usepackage{psfrag,rotating}
\usepackage{t1enc}
\usepackage{cite}
\usepackage{esint}
\usepackage{url}
\usepackage[]{color}
\usepackage{xcolor}
\usepackage{mathtools}
\usepackage[T1]{fontenc} 
\usepackage{lmodern}  
\usepackage{comment}
\usepackage{multirow}
\usepackage{longtable}
\usepackage{caption}
\usepackage{subcaption}
\usepackage{float}
\usepackage{orcidlink}

\allowdisplaybreaks[1]

\numberwithin{equation}{section}

\newtheorem{lemma}{Lemma}
\newtheorem*{lemma*}{Lemma}
\newtheorem{proposition}{Proposition}
\newtheorem{theorem}{Theorem}

\newtheorem{definition}{Definition}
\newtheorem*{remark}{Remark}

\begin{document}

\title{Well-posedness of an evaporation model\\
for a spherical droplet exposed to an air flow}

\author{Eberhard B\"ansch\textsuperscript{a}, Martin Do{\ss}\textsuperscript{a,}\footnote{Correspondence:  \href{mailto:martin.md.doss@fau.de}{martin.md.doss@fau.de}}~\,\orcidlink{0000-0002-5842-6376}, Carsten Gr{\"a}ser\textsuperscript{a}\,\orcidlink{0000-0003-4855-8655}, and Nadja Ray\textsuperscript{a,b}\,\orcidlink{0000-0002-9596-953X}\vspace{10pt}\\
\normalsize \textsuperscript{a}Department of Mathematics, Friedrich--Alexander University\\
\normalsize Erlangen--N{\"u}rnberg, Erlangen, Germany\\
\normalsize \textsuperscript{b}Mathematical Institute for Machine Learning and Data Science, Catholic University\\
\normalsize Eichst{\"a}tt--Ingolstadt, Ingolstadt, Germany}

\date{}

\maketitle 

\begin{abstract}
    In this paper, we address the well-posedness of an evaporation model for a spherical liquid droplet taking into account the convective impact of an air flow in the ambient gas phase. From a mathematical perspective, we are dealing with a coupled ODE--PDE system for the droplet radius, the temperature distribution, and the vapor concentration.  The nonlinear coupling arises from the evaporation rate modeled by the Hertz--Knudsen equation. Under physically meaningful assumptions, we prove existence and uniqueness of a weak solution until the droplet has evaporated completely. Numerical simulations are performed to illustrate how different air flows affect the evaporation process.\vspace{10pt}\vspace{10pt} 
\end{abstract}

\section{Introduction} \label{sec1}

Evaporation plays an important role in a variety of practical applications ranging from weather forecasting~\cite{wang12} to the production of pharmaceutical powders\cite{vehring08}. Especially when convective currents complicate the temperature and vapor mass distributions, it is important to understand their impact on the evaporation process. Mathematical models can be powerful tools to unravel the complex physics behind convective evaporation provided that the heat and mass fluxes across the liquid--gas interface are captured correctly. Depending on the characteristic length and time scales, there are two different modeling approaches: If the volatile liquid evaporates slowly, one usually assumes thermodynamic equilibrium in the sense that the saturated equals the actual vapor pressure at the liquid--gas interface. In general, however, the saturated and the actual vapor pressure do not coincide and their difference determines the evaporation rate. Taking into account this imbalance, non-equilibrium evaporation models are physically more accurate. Especially the Hertz--Knudsen equation is widely used to model and simulate evaporation \cite{holyst15}. As a nonlinear boundary condition, it couples the heat and mass transfer within and around the volatile liquid. Among all possible applications, especially the evaporation of small droplets was found to be well described by the Hertz--Knudsen equation \cite{zientara13}. Our aim is to study the mathematical well-posedness of the related evaporation model for a single droplet taking into account the convective impact of an air flow in the ambient gas phase.\\
The following literature is related to our problem. First of all, the quasi-stationary evaporation of a spherical droplet into stagnant air is well described analytically by the $d^2$-law. As suggested by its name, the $d^2$-law states that the squared diameter (or radius) of the droplet decreases linearly over time \cite{law06}. The same behavior follows from the model in \cite{anderson07} which also accounts for non-equilibrium effects. From a mathematical perspective, single droplet evaporation (and condensation) can be regarded as a generalized Stefan problem \cite{friedman60}. The latter is widely known to be well-posed in one and more spatial dimensions \cite{friedman76, rubinstein71, visintin08, escher03}. The main difference between single droplet evaporation and the classical Stefan problem arises from the fact that the saturated vapor pressure depends on the temperature at the droplet surface which is neglected in \cite{friedman60}. The authors of \cite{pruess11, pruess12} consider a more general model for two-phase flow with phase transition derived from the balances of mass, momentum, and energy in a thermodynamically consistent way. Their existence result applies to nearly flat interfaces and small initial data. Finally, also the cell problems whose well-posedness is studied in \cite{gahn23} are closely related to our topic despite their application being different. Since the upscaled model in \cite{gahn23} describes mineral dissolution and precipitation in a porous medium, each cell represents a single spherical grain whose radius is governed by a nonlinear ordinary differential equation. More general cell geometries are studied in \cite{gaerttner23} using a level-set approach to prove the stability of the corresponding diffusion and permeability tensors. Even if the interface conditions for mineral dissolution and droplet evaporation appear to be similar, it should be pointed out that the latter are generally more complicated. While dissolution can be regarded as an isothermal process, evaporation causes a discontinuity of the normal heat flux which is commonly known as evaporative cooling. Therefore, the evaporation rate determines both the mass and the heat flux across the liquid--gas interface.\\
Our aim is to study the mathematical well-posedness of a non-equilibrium evaporation model for a single spherical droplet which is exposd to an air flow in the ambient gas phase. The considered evaporation model consists of an ordinary differential equation for the droplet radius and two convection--diffusion equations for the temperature and vapor mass distributions. The mathematical challenges of the resulting ODE--PDE system mainly arise from its nonlinear coupling at the droplet surface and the evolution of the latter as a free boundary. More precisely, the evaporation rate is computed from the thermodynamic non-equilibrium between the saturated and the actual vapor pressure at the droplet surface via the Hertz--Knudsen equation. To handle the resulting nonlinear boundary condition for the heat and vapor mass flux, we apply the method of upper and lower solutions. Regarding the radius of the droplet as its time-dependent characteristic length, we resale our model accordingly to immobilize the free phase boundary. By that, the interface velocity can be treated as an additional flux term in the transport equations for the temperature and vapor mass distributions. Finally, we apply Banach's fixed-point theorem to prove the existence of a unique weak solution.\\
The paper is structured as follows. In Section \ref{sec_model}, we introduce our mathematical model for the convective evaporation of a single droplet. In Section \ref{sec_well_posedness}, the latter is shown to admit a unique maximal weak solution as outlined above. The numerical simulations performed in Section \ref{sec_numerical_examples} illustrate the convective impact of the ambient air flow on the evaporation rate. Finally, Section \ref{sec_conclusion} concludes our paper.

\section{Droplet evaporation model} \label{sec_model}

In this section, we present our mathematical model for the convective evaporation of a single spherical droplet exposed to an air flow in the ambient gas phase. As illustrated in Figure \ref{fig:illustration}, 
\begin{figure}
    \centering
    \includegraphics[width=0.7\textwidth]{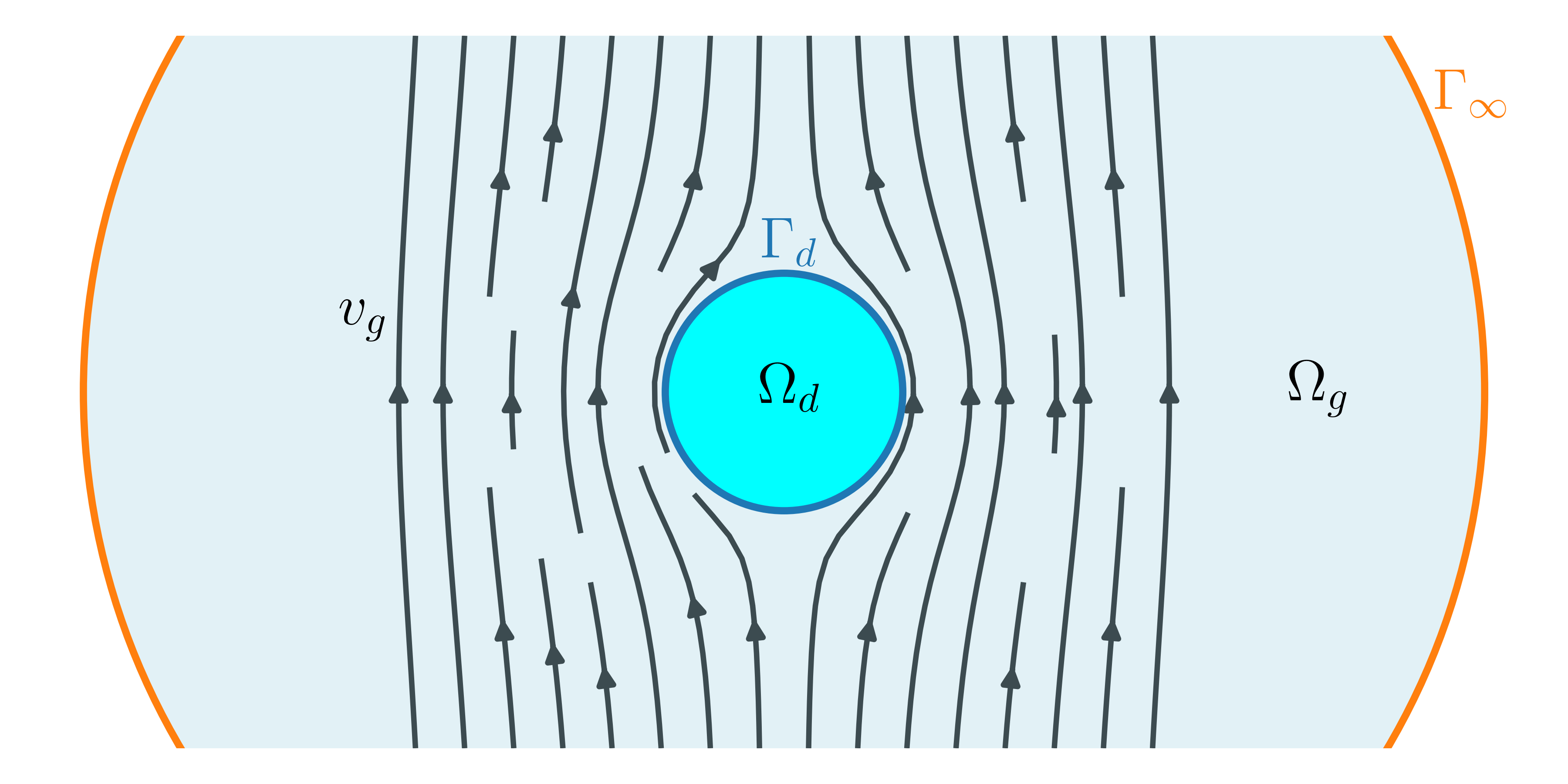}
    \caption{Schematic illustration of the evaporating droplet exposed to an air flow.}
    \label{fig:illustration}
\end{figure}
the droplet surface is denoted by $\Gamma_d \coloneqq \partial \Omega_d$ while $\Omega_d \subset \mathbb{R}^3$ represents its liquid interior. The total droplet volume $V_d \coloneqq |\Omega_d|$ is then governed by the ordinary differential equation
\begin{equation} \label{droplet_volume_evolution}
    \rho_d \frac{d V_d}{dt} = - \int_{\Gamma_d} J \, d\sigma
\end{equation}
with the mass density $\rho_d$ of the volatile liquid and the evaporation rate $J$ per surface area. The ambient gas phase $\Omega_g$ is usually not stagnant. In many practical applications, either natural or forced convection can have a significant impact on the evaporation process. Let $v_g$ be the corresponding velocity field and $n$ the outer unit normal of the gas phase. The heat and vapor mass transport is then governed by the following system of convection--diffusion equations
\begin{align}
    \rho_d C_{p,d}\partial_t T_d &= \nabla \cdot (k_d \nabla T_d) &&\text{in } \Omega_d,\\
    \rho_g C_{p,g}\partial_t T_g &= \nabla \cdot (k_g \nabla T_g - \rho_g C_{p,g} T_g v_g) &&\text{in } \Omega_g,\\
    \partial_t \rho_v &= \nabla \cdot (D_v \nabla \rho_v - \rho_v v_g) &&\text{in } \Omega_g
\end{align}
subject to the interface flux conditions
\begin{align}
    (k_g \nabla T_g - k_d \nabla T_d) \cdot n &= - \Lambda J &&\text{on } \Gamma_d,\\
    D_v \nabla \rho_v \cdot n &= J &&\text{on } \Gamma_d
\end{align}
where $\rho_g$ denotes the total mass density of the gas phase, $\rho_v$ the vapor mass density, $T_{d/g}$ the droplet/gas temperature, $C_{p,d/g}$ the respective specific heat capacity, $k_{d/g}$ the respective thermal conductivity, $D_v$ the vapor diffusion coefficient, and $\Lambda$ the specific latent heat of vaporization.\\
The evaporation rate $J$ is modeled by the Hertz--Knudsen equation. It relates $J$ to the thermodynamic non-equilibrium between the saturated and the actual vapor pressure at the liquid--gas interface. In its original formulation, the Hertz--Knudsen equation reads as follows \cite{holyst15}
\begin{equation} \label{Hertz_Kndusen}
    J = \sqrt{\frac{M_d}{2 \pi \mathcal{R} T_d}} (p_\textit{sat} - p_v)
\end{equation}
with the molar mass $M_d$ of the volatile liquid, the ideal gas constant $\mathcal{R}$, the droplet temperature $T_d$, the saturated vapor pressure $p_\textit{sat}$, and the actual vapor pressure $p_v$ at the liquid--gas interface. It should be mentioned that $p_v$ and $\rho_v$ are related to each other via the ideal gas law $p_v M_d = \rho_v \mathcal{R} T_g$. The saturated vapor pressure, on the other hand, is determined by the interface temperature via the Clausius--Clapeyron relation. In case of water, the latter is well approximated by the Tetens equation \cite{murray67}
\begin{equation} \label{Tetens_equation}
	\frac{p_\textit{sat}}{610.78\,\text{Pa}} = \exp \left( \frac{17.27\,T_g}{T_g + 237.3\text{\,\textdegree C}} \right)
\end{equation}
with $T_g$ in degree Celsius. The continuity of temperature naturally implies $T_d = T_g$ on $\Gamma_d$ which serves as an additional boundary condition. Far away from the droplet, the temperature and vapor mass distributions are supposed to be consistent with the drying conditions. Their radial limits are given by
\begin{equation}
    \lim_{|x| \to \infty} T_g(t,x) = T_\infty \text{\qquad and \qquad} \lim_{|x| \to \infty} \rho_v(t,x) = \rho_\infty
\end{equation}
where $T_\infty$ denotes the temperature and $\rho_\infty$ the vapor mass density of the drying air. It should be pointed out that $T_\infty$ and $\rho_\infty$ are both regarded as constants.\\
The presented evaporation model is highly relevant for a plethora of practical applications such as spray drying or air conditioning \cite{semenov11}. However, to the best of our knowledge, its mathematical well-posedness has not yet been verified in the existing literature. In this paper, we show existence and uniqueness of a weak solution under the following assumptions and simplifications:
\begin{itemize}
    \item[(A1)] The droplet remains spherical throughout its evaporation process. Its radius $R_d$ is governed by
    \begin{equation}
        \rho_d \frac{d R_d}{dt} = - \frac{1}{4 \pi R_d^2} \int_{\Gamma_d} J \, d\sigma
    \end{equation}
    which follows from \eqref{droplet_volume_evolution} by inserting the formula $V_d = 4/3 \pi R_d^3$ for the droplet volume.
    \item[(A2)] As illustrated in Figure \ref{fig:illustration}, the ambient gas phase $\Omega_g$ is truncated far away from the droplet such that its rescaled domain $\Omega_g^* \coloneqq \Omega_g / R_d$ becomes a spherical shell with fixed boundaries.
    \item[(A3)] The drying conditions are sufficiently moderate such that
    \begin{equation}
        |k_d \nabla T_d \cdot n| \ll |k_g \nabla T_g \cdot n| \text{\qquad on } \Gamma_d
    \end{equation} 
    allows us to neglect thermal variations inside the droplet. In other words, only the heat and mass transfer in the gas phase is assumed to determine the evaporation rate.
    \item[(A4)] The Hertz--Knudsen equation \eqref{Hertz_Kndusen} for the evaporation rate is approximated as follows
    \begin{equation} \label{Hertz_Knudsen_approx}
        J = \sqrt{\frac{\mathcal{R} T_g}{2 \pi M_d}} (\rho_\textit{sat} - \rho_v) \approx C (\rho_\textit{sat} - \rho_v)
    \end{equation}
    neglecting the temperature dependence of $C > 0$ and simply regarding it as a positive constant. Motivated by the Tetens equation \eqref{Tetens_equation}, the saturated vapor mass density $\rho_\textit{sat} \coloneqq p_\textit{sat} M_w / (\mathcal{R} T_g)$ is assumed to admit a continuous derivative with respect to the gas temperature satisfying
    \begin{equation}
        0 \leq \frac{d \rho_\textit{sat}}{d T_g} \leq L 
    \end{equation}
    for some Lipschitz constant $L > 0$.
    \item[(A5)] The drying conditions satisfy $\rho_\infty \leq \rho_\textit{sat} (T_\infty) \eqqcolon \rho_*$ which guarantees evaporation instead of condensation. Moreover, there exists a temperature $T_* \leq T_\infty$ such that the corresponding saturated vapor mass density equals $\rho_\textit{sat} (T_*) = \rho_\infty$.
    \item[(A6)] The air flow is described by a quasi-stationary incompressible velocity field $v_g \in H^1(\Omega_g, \mathbb{R}^3)$ satisfying $\nabla \cdot v_g = 0$ in $\Omega_g$ and $v_g \cdot n = 0$ on $\Gamma_d$. It depends continuously on the droplet radius in the following sense: For each $\eta > 0$ there exists a Lipschitz constant $L_\eta > 0$ such that 
    \begin{equation} \label{air_flow_continuity}
        \sup_{x\,\in\,\Omega_g^*} \big| v_g^1 (x R_d^1) - v_g^2 (x R_d^2) \big| \leq L_\eta \big| R_d^1 - R_d^2 \big|
    \end{equation}
    holds for all droplet radii $R_d^1, R_d^2 > \eta$ where $v_g^1$ and $v_g^2$ denote the corresponding velocity fields for the air flow in the ambient gas phase. In other words, the rescaled velocity field $v_g^*$ defined below is uniquely determined by the droplet radius $R_d$ and \eqref{air_flow_continuity} guarantees the Lipschitz continuity of the mapping $R_d \mapsto v_g^* \in L^\infty (\Omega_g^*; \mathbb{R}^3)$ if the droplet radius $R_d$ is bounded away from zero.
    \item[(A7)] Without loss of generality, all physical constants are assumed to be one for the ease of presentation. However, keep in mind that $T_g$ and $\rho_v$ do not satisfy the same convection--diffusion equation in the sense that their coefficients are different!
\end{itemize}
Regarding the radius $R_d$ of the droplet as its time-dependent characteristic length, we rescale our model accordingly. Define $x^* \coloneqq x / R_d \in \Omega_g^*$ for all spatial coordinates $x \in \Omega_g$. The time derivatives of the gas temperature $T_g$ and its rescaled counterpart $T_g^* (t, x^*) \coloneqq T_g (t, x)$ are then related to each other via
\begin{equation} \label{rescaled_time_derivative}
    \partial_t T_g (t, x) = \partial_t T_g^* (t, x^*) + \nabla^* T_g^* (t, x^*) \cdot \partial_t x^* = \partial_t T_g^* (t, x^*) - \frac{\dot{R}_d}{R_d} \nabla^* T_g^* (t, x^*) \cdot x^*
\end{equation}
where $\nabla^* = R_d \nabla$ denotes the rescaled spatial gradient. Notice that \eqref{rescaled_time_derivative} follows directly from the chain rule. The same identity holds for the rescaled counterpart $\rho_v^* (t, x^*) \coloneqq \rho_v (t, x)$ of the vapor mass density. Altogether, the rescaled simplified version of our single droplet evaporation model whose well-posedness will be studied in Section \ref{sec_well_posedness} reads as follows
\begin{align}
    \dot{R}_d &= - \frac{1}{4 \pi} \int_{\Gamma_d^*} J(T_g^*, \rho_v^*) \, d\sigma, &&~\label{radius_evolution} \tag{E1}\\
    \partial_t T_g^* &= \frac{1}{R_d^2} \Delta^* T_g^* - \frac{1}{R_d} \nabla^* T_g^* \cdot ( v_g^* - \dot{R}_d x^* ) &&\text{in } \Omega_g^*, \label{strong_formulation_T} \tag{E2}\\
    \partial_t \rho_v^* &= \frac{1}{R_d^2} \Delta^* \rho_v^* - \frac{1}{R_d} \nabla^* \rho_v^* \cdot ( v_g^* - \dot{R}_d x^* ) &&\text{in } \Omega_g^*, \label{strong_formulation_rho} \tag{E3}\\
    \frac{1}{R_d} \nabla^* T_g^* \cdot n^* &= - J(T_g^*, \rho_v^*) &&\text{on } \Gamma_d^*, \tag{E4}\\
    \frac{1}{R_d} \nabla^* \rho_v^* \cdot n^* &= J(T_g^*, \rho_v^*) &&\text{on } \Gamma_d^*, \tag{E5}\\
    T_g^* &= T_\infty \phantom{\frac{1}{R_d}} &&\text{on } \Gamma_\infty^*, \tag{E6}\\
    \rho_v^* &= \rho_\infty \phantom{\frac{1}{R_d}} &&\text{on } \Gamma_\infty^* \label{drying_condition_rho} \tag{E7}
\end{align}
with the Nemytskii operator $J(T_g^*, \rho_v^*) \coloneqq \rho_\textit{sat}(T_g^*) - \rho_v^*$ for the evaporation rate. Recall that the boundaries $\Gamma_d^* \coloneqq \{x^* \in \mathbb{R}^3 \, : \, |x^*| = 1\}$ and  $\Gamma_\infty^* = \partial \Omega_g^* \setminus \Gamma_d^*$ of the rescaled gas phase are now fixed. In return, the interface velocity now appears in the transport equations of our model. For the ease of presentation, the sub- and superscripts `$d$', `$g$', `$v$', and `$*$' will be omitted in the following.

\section{Well-posedness of the droplet evaporation model} \label{sec_well_posedness}

In this section, we show that the presented droplet evaporation model \eqref{radius_evolution}--\eqref{drying_condition_rho} admits a unique weak solution taking into account the assumptions (A1)--(A7) from above. From a mathematical perspective, we are dealing with a coupled ODE--PDE system for the droplet radius, the temperature distribution, and the vapor mass concentration.  The nonlinear coupling arises from the evaporation rate being modeled by the Hertz--Knudsen equation.\\
As mentioned above, the well-posedness of our single droplet evaporation model will be studied in the context of weak solutions. The underlying Lebesgue and Sobolev spaces are denoted by $L^2 (\Omega)$ and $H^1 (\Omega) = W^{1,2} (\Omega)$, respectively. Let $\langle \cdot, \cdot \rangle$ be the pairing of $X \coloneqq \{ \varphi \in H^1(\Omega) \,:\, \varphi = 0 \text{ on } \Gamma_\infty \}$ and its dual space $X^*$. The droplet evaporation problem whose governing equations and boundary conditions \eqref{radius_evolution}--\eqref{drying_condition_rho} were derived in Section \ref{sec_model} admits the following weak formulation:

\begin{definition}[Weak solution of the droplet evaporation problem] \label{Def:weak_solution}
    Let $R_0 > 0$ and $T_0, \rho_0 \in L^2(\Omega)$ be given. The triple $(R, T, \rho)$ is called a weak solution of the droplet evaporation problem \eqref{radius_evolution}--\eqref{drying_condition_rho} if
    \begin{itemize}
        \item $R \in H^1(I)$ satisfies $R(0) = R_0$ and the evolution equation
        \begin{equation} \label{droplet_radius_evolution}
            \dot{R} (t) = - \frac{1}{4 \pi} \int_{\Gamma} J(T, \rho) \, d\sigma \tag{P1}
        \end{equation}
        for a.e. $t \in I$,
        \item $T, \rho \in L^2(I,H^1(\Omega))$ with $\partial_t T, \partial_t \rho \in L^2(I,X^*)$ satisfy the weak formulations
        \begin{equation} \label{Eq:1.2}
            \langle \partial_t T, \varphi \rangle + \int_{\Omega} \frac{1}{R^2} \nabla T \cdot \nabla \varphi + \frac{1}{R} \nabla T \cdot ( v - \dot{R} x ) \varphi \, dx
            + \frac{1}{R} \int_{\Gamma} J (T, \rho) \varphi \, d\sigma = 0, \tag{P2}
        \end{equation}
        \begin{equation} \label{Eq:1.1}
            \langle \partial_t \rho, \varphi \rangle + \int_{\Omega} \frac{1}{R^2} \nabla \rho \cdot \nabla \varphi + \frac{1}{R} \nabla \rho \cdot ( v - \dot{R} x ) \varphi \, dx
            - \frac{1}{R} \int_{\Gamma} J (T, \rho) \varphi \, d\sigma = 0 \tag{P3}
        \end{equation}
        for all test functions $\varphi \in X$,
        \item the initial and drying conditions
        \begin{gather}
            T(0) = T_0 \text{\qquad and \qquad} \rho(0) = \rho_0 \text{\qquad in } \Omega, \tag{P4}\\
            T = T_\infty \text{\qquad and \qquad} \rho = \rho_\infty \text{\qquad on } \Gamma_\infty \tag{P5} \label{drying_conditions}
        \end{gather}
        are fulfilled. 
    \end{itemize}
\end{definition}

As formulated in Theorem \ref{theorem_maximal_existence}, our main result will be the existence of a unique weak solution in the sense of Definition \ref{Def:weak_solution} until the droplet has evaporated completely. Formally speaking, the strategy of our proof can be outlined as follows. First, we consider the droplet radius to be given and show that the decoupled evaporation problem \eqref{Eq:1.2}--\eqref{drying_conditions} admits a unique weak solution. In Proposition \ref{H1_stability}, the resulting solution operator $S(R) \coloneqq (T, \rho)$ is shown to be Lipschitz continuous. The coupled ODE--PDE system \eqref{droplet_radius_evolution}--\eqref{drying_conditions} with unknown droplet radius is then reformulated as a fixed-point problem. Notice that $R$ is a fixed-point of the following Volterra operator
\begin{equation} \label{volterra}
    \mathcal{T} (R) (t) \coloneqq R_0 - \frac{1}{4 \pi} \int_0^t \int_{\Gamma} J (S(R)) \, d\sigma \, d \tau
\end{equation}
if $(R, T, \rho)$ is a weak solution of the droplet evaporation problem. On the other hand, if $R \in H^1(I)$ satisfies $\mathcal{T} (R) = R$ then $(R, S(R))$ solves the droplet evaporation problem in the sense of Definition \ref{Def:weak_solution}. Therefore, finding a solution to the latter is equivalent to a fixed-point problem. In Section \ref{sec:short_time_existence}, we show that $\mathcal{T}$ is a well-defined contraction on the following set of admissible droplet radii
\begin{equation}
\Sigma_{t_*} \coloneqq \big\{R \in H^1(I) \,:\, R(0) = R_0 \text{ and } |\dot{R}(t)| \leq J(T_\infty, \rho_\infty) \text{ for a.e. } t \in I \big\}
\end{equation}
provided that the underlying time interval $I \coloneqq [0, t_*)$ is sufficiently small. The existence of a fixed-point $R \in \Sigma_{t_*}$ with $\mathcal{T} (R) = R$ then follows from Banach's fixed-point theorem. Finally, we show that the time interval $I$ of the weak solution can be extended until the droplet has evaporated completely.

\subsection{Preliminaries}

Throughout the entire paper, the constant $C > 0$ is meant to be generic in the sense that its value can be redefined. Unless stated otherwise, it only depends on constant coefficients and the spatial domain $\Omega$ of our model, but not on the time interval $I$ of its solution. For the ease of presentation, we define $\| \cdot \|_{2, \Omega} \coloneqq \| \cdot \|_{L^2(\Omega)}$ to abbreviate the $L^2(\Omega)$-norm. The following interpolation inequality will be instrumental to handle boundary terms.

\begin{lemma}\label{Lemma:1}
Let $\Omega \subset \mathbb{R}^3$ be a bounded $C^2$-domain. Then $u \in H^1(\Omega)$ satisfies 
\begin{equation} \label{Ehlring_inequality}
    \| u \|_{2, \partial \Omega}^2 \leq \varepsilon \| \nabla u \|_{2, \Omega}^2 + \frac{C}{\varepsilon} \| u \|_{2, \Omega}^2
\end{equation}
for all $\varepsilon \in (0, 1)$. The constant $C > 0$ only depends on $\Omega$.
\end{lemma}

From \cite[Lemma 7.9]{herz14} we adopt the following proof:

\begin{proof}
    First of all, the continuity of the trace operator\cite[Theorem 7.58]{adams75} implies
    \begin{equation}
        \| u \|_{2, \partial \Omega} \leq C \| u \|_{H^{1/2}(\Omega)} 
    \end{equation}
    for some $C > 0$. From $H^1(\Omega) \hookrightarrow H^{1/2}(\Omega) \hookrightarrow L^2(\Omega)$ we infer the interpolation inequality
    \begin{equation}
        \| u \|_{H^{1/2}(\Omega)}^2 \leq C \| u \|_{H^1(\Omega)} \| u \|_{2, \Omega}
    \end{equation}
    to be found in \cite[Lemma 7.16]{adams75}. Using Young's inequality, we finally obtain
    \begin{equation}
        \| u \|_{2, \partial \Omega}^2 \leq C \| u \|_{H^1(\Omega)} \| u \|_{2, \Omega} \leq \varepsilon \| u \|_{H^1(\Omega)}^2 + \frac{C}{\varepsilon} \| u \|_{2, \Omega}^2 \leq \varepsilon \| \nabla u \|_{2, \Omega}^2 + \frac{C}{\varepsilon} \| u \|_{2, \Omega}^2
    \end{equation}
    where the last inequality follows from $\varepsilon < 1$.
\end{proof}

\begin{remark}
    It should be pointed out that Lemma \ref{Lemma:1} holds for general smooth geometries. In our case, the boundary $\Gamma$ represents the rescaled surface of a spherical droplet. Moreover, Lemma \ref{Lemma:1} will only be applied to functions $u \in H^1(\Omega)$ satisfying $u = 0$ on $\Gamma_\infty$. Under these additional constraints, one obtains the interpolation inequality \eqref{Ehlring_inequality} also directly from the divergence theorem
\begin{equation}
\begin{gathered}
    \| u \|_{2,\Gamma}^2 = - \int_{\Omega} \nabla \cdot \left( u^2 x \right) \, dx = - \int_{\Omega} 2 u \nabla u \cdot x + 3 u^2 \, dx\\
    \leq 2 C \int_{\Omega} |u| |\nabla u| \, dx \leq \varepsilon \| \nabla u \|_{2,\Omega}^2 + \frac{C}{\varepsilon} \| u \|_{2,\Omega}^2
\end{gathered}
\end{equation}
taking into account $x \cdot n = -1$ for all $x \in \Gamma$ and $|x| \leq C$ for all $x \in \Omega$.
\end{remark}

\subsection{Well-posedness for given droplet radius}

In this section, we apply the method of upper and lower solutions to show that the decoupled droplet evaporation problem \eqref{Eq:1.2}--\eqref{drying_conditions} admits a unique weak solution if $R \in \Sigma_{t_*}$ is given. The idea is to replace the original nonlinear boundary conditions at the droplet surface by the following linear ones \cite{pao92}
\begin{equation}
\begin{gathered}
    \frac{1}{R} \nabla T^{k} \cdot n + L T^{k} = - J^{k-1} + L T^{k-1},\\
    \frac{1}{R} \nabla \rho^{k} \cdot n + \rho^{k} = J^{k-1} + \rho^{k-1}
\end{gathered}
\end{equation}
where $L$ denotes the Lipschitz constant of $\rho_\textit{sat}$ and $J^{k-1} \coloneqq J (T^{k-1}, \rho^{k-1})$ the explicit evaporation rate. By that, we obtain an iteration process $(T^{k-1}, \rho^{k-1}) \mapsto (T^{k}, \rho^{k})$ converging to the desired weak solution.

\begin{lemma} \label{lemma_well_posedness_iteration_process}
Let $T^{k-1}, \rho^{k-1} \in L^2(I, H^1(\Omega))$ and $R \in \Sigma_{t_*}$ be given. Further assume that $R$ is bounded away from zero in the sense that $R(t) \geq \eta > 0$ for a.e. $t \in I$. Then, there exist unique weak solutions
\begin{equation}
    T^k, \rho^k \in L^2(I, H^1(\Omega)) \text{\qquad with \qquad} \partial_t T^k, \partial_t \rho^k \in L^2(I, X^*)
\end{equation}
of the linear iteration process governed by the weak formulations
\begin{equation}\label{weak_iteration_process_T}
    \langle \partial_t T^k, \varphi \rangle + \int_{\Omega} \frac{1}{R^2} \nabla T^{k} \cdot \nabla \varphi + \frac{1}{R} \nabla T^{k} \cdot ( v - \dot{R} x ) \varphi \, dx + \frac{1}{R} \int_{\Gamma} L T^k \varphi \, d\sigma = \frac{1}{R} \int_{\Gamma} (L T^{k-1} - J^{k-1}) \varphi \, d\sigma,
\end{equation}
\begin{equation}\label{weak_iteration_process_rho}
    \langle \partial_t \rho^k, \varphi \rangle + \int_{\Omega} \frac{1}{R^2} \nabla \rho^{k} \cdot \nabla \varphi + \frac{1}{R} \nabla \rho^{k} \cdot ( v - \dot{R} x ) \varphi \, dx + \frac{1}{R} \int_{\Gamma} \rho^k \varphi \, d\sigma = \frac{1}{R} \int_{\Gamma} (\rho^{k-1} + J^{k-1}) \varphi \, d\sigma
\end{equation}
for all test functions $\varphi \in X$ such that the initial and drying conditions
\begin{gather}
    T^k(0) = T_0 \text{\qquad and \qquad} \rho^k(0) = \rho_0 \text{\qquad in } \Omega,\\
    T^k = T_\infty \text{\qquad and \qquad} \rho^k = \rho_\infty \text{\qquad on } \Gamma_\infty
\end{gather}
are fulfilled. Moreover, the weak solutions satisfy the following energy estimates
\begin{gather}
    \sup_{t \in I} \left\| T^k - T_\infty \right\|_{2,\Omega}^2 + \left\| \nabla T^k \right\|_{2,I\times\Omega}^2 \leq C \left[ \left\| T_0 - T_\infty \right\|_{2, \Omega}^2 + \left\| L (T^{k-1} - T_\infty) - J^{k-1} \right\|_{2, I \times \Gamma}^2 \right], \label{energy_estimate1_T}\\
    \sup_{t \in I} \left\| \rho^k - \rho_\infty \right\|_{2,\Omega}^2 + \left\| \nabla \rho^k \right\|_{2,I\times\Omega}^2 \leq C \left[ \left\| \rho_0 - \rho_\infty \right\|_{2, \Omega}^2 + \left\| \rho^{k-1} - \rho_\infty + J^{k-1} \right\|_{2, I \times \Gamma}^2 \right], \label{energy_estimate1_rho}\\
    \left\| \partial_t T^k \right\|_{L^2(I, X^*)}^2 \leq C \left[ \left\| T_0 - T_\infty \right\|_{2, \Omega}^2 + \left\| L (T^{k-1} - T_\infty) - J^{k-1} \right\|_{2, I \times \Gamma}^2 \right], \label{energy_estimate2_T}\\
    \left\| \partial_t \rho^k \right\|_{L^2(I, X^*)}^2 \leq C \left[ \left\| \rho_0 - \rho_\infty \right\|_{2, \Omega}^2 + \left\| \rho^{k-1} - \rho_\infty + J^{k-1} \right\|_{2, I \times \Gamma}^2 \right] \label{energy_estimate2_rho}
\end{gather}
where the constant $C > 0$ depends on the positive lower bound $\eta$ of the droplet radius.
\end{lemma}

\begin{proof}
    The unique existence of the weak solutions $T^k$ and $\rho^k$ follows from Galerkin's method as carried out in \cite{evans10, lady68} for similar linear parabolic problems. Testing \eqref{weak_iteration_process_T} with $\varphi = T^k - T_\infty \eqqcolon w^k$ implies
    \begin{equation}
        \frac{1}{2} \frac{d}{dt} \| w^k \|_{2,\Omega}^2 + \frac{1}{R^2} \| \nabla w^k \|_{2,\Omega}^2 - \frac{\dot{R}}{R} \int_\Omega (\nabla w^k \cdot x) w^k \, dx + \frac{L}{R} \| w^k \|_{2,\Gamma}^2 = \frac{1}{R} \int_\Gamma (L w^{k-1} - J^{k-1}) w^k \, d\sigma
    \end{equation}
    since $v$ is divergence-free and tangential at the droplet surface. Applying H\"older's inequality yields
    \begin{equation}
        \frac{1}{2} \frac{d}{dt} \| w^k \|_{2,\Omega}^2 + \| \nabla w^k \|_{2,\Omega}^2 \leq C \big[ \| \nabla w^k \|_{2,\Omega} \| w^k \|_{2,\Omega} + \| L w^{k-1} - J^{k-1} \|_{2,\Gamma} \| w^k \|_{2,\Gamma} \big]
    \end{equation}
    due to $|\dot{R}| \leq J_\infty$ and $R$ being bounded away from zero. The estimate
    \begin{equation}
        \frac{1}{2} \frac{d}{dt} \| w^k \|_{2,\Omega}^2 + \| \nabla w^k \|_{2,\Omega}^2 \leq C \Big[ \varepsilon \| \nabla w^k \|_{2,\Omega}^2 + \frac{1}{\varepsilon} \| w^k \|_{2,\Omega}^2 + \| L w^{k-1} - J^{k-1} \|_{2,\Gamma}^2 \Big]
    \end{equation}
    then follows from Young's inequality with $\varepsilon > 0$ and the interpolation inequality from Lemma \ref{Lemma:1}. The choice $\varepsilon = 1/(2C)$ allows us to absorb the $L^2$-norm of the gradient. Applying Gr\"onwall's lemma to the resulting inequality
    \begin{equation}
        \frac{d}{dt} \| w^k \|_{2,\Omega}^2 + \| \nabla w^k \|_{2,\Omega}^2 \leq C \big[ \| w^k \|_{2,\Omega}^2 + \| L w^{k-1} - J^{k-1} \|_{2,\Gamma}^2 \big]
    \end{equation}
    implies \eqref{energy_estimate1_T} while \eqref{energy_estimate1_rho} can be derived analogously. On the other hand, we obtain
    \begin{equation} \label{dual_estimate_T}
        \big| \langle \partial_t T^k, \varphi \rangle \big| \leq C \big[ \| \nabla w^k \|_{2,\Omega} + \| L w^{k-1} - J^{k-1} \|_{2,\Gamma} \big] \| \varphi \|_X
    \end{equation}
    directly from \eqref{weak_iteration_process_T} by applying H\"older's inequality since (A6) implies $v \in L^\infty (\Omega)$. The estimate
    \begin{equation} \label{dual_estimate_T_final}
        \| \partial_t T^k \|_{X^*} \leq C \big[ \| \nabla w^k \|_{2,\Omega} + \| L w^{k-1} - J^{k-1} \|_{2,\Gamma} \big]
    \end{equation}
    then follows from the fact that \eqref{dual_estimate_T} holds for all $\varphi \in X$. Inserting \eqref{energy_estimate1_T} into the resulting inequality
    \begin{equation}
        \| \partial_t T^k \|_{L^2(I, X^*)}^2 \leq C \big[ \| \nabla w^k \|_{2,I \times \Omega}^2 + \| L w^{k-1} - J^{k-1} \|_{2,I\times\Gamma}^2 \big]
    \end{equation}
    finally implies \eqref{energy_estimate2_T} while \eqref{energy_estimate2_rho} can be derived analogously.
\end{proof}

We are now in the position to show the well-posedness of \eqref{Eq:1.2}--\eqref{drying_conditions} for given droplet radius. Recall from assumption (A5) that the drying conditions are required to satisfy $\rho_\textit{sat} (T_\infty) \geq \rho_\infty$ such that the droplet actually evaporates. In the following proof, the constant tuples $(T_*, \rho_\infty)$ and $(T_\infty, \rho_*)$ will serve as lower and upper bounds for the weak solution $(T, \rho)$ of the decoupled droplet evaporation problem.

\begin{proposition} \label{existence_Rd_given}
    Assume $R \in \Sigma_{t_*}$ to be given and bounded away from zero. Further suppose that the initial values $T_0, \rho_0 \in L^2 (\Omega)$ are bounded by the drying conditions and constants from \textnormal{(A5)} as follows
    \begin{equation}\label{initial_lower_and_upper_bounds}
        T_* \leq T_0 (x) \leq T_\infty \text{\qquad and \qquad} \rho_\infty \leq \rho_0 (x) \leq \rho_*
    \end{equation}
    for a.e. $x \in \Omega$. Then, the decoupled droplet evaporation problem \eqref{Eq:1.2}--\eqref{drying_conditions} admits a unique weak solution 
    \begin{equation}
        (T, \rho) \in L^2 (I, H^1(\Omega))^2 \text{\qquad with \qquad} (\partial_t T, \partial_t \rho) \in L^2 (I, X^*)^2
    \end{equation}
    satisfying the pointwise lower and upper bounds
    \begin{equation} \label{lower_and_upper_bounds}
        T_* \leq T (t,x) \leq T_\infty \text{\qquad and \qquad} \rho_\infty \leq \rho (t,x) \leq \rho_*
    \end{equation}
    for a.e. $(t, x) \in I \times \Omega$.
\end{proposition}

\begin{proof}    
    Following \cite{pao92, evans10} we apply the method of upper and lower solutions. Therefore, let us consider the iteration process $(T^{k-1}, \rho^{k-1}) \mapsto (T^{k}, \rho^{k})$ defined by the linear parabolic problems from Lemma \ref{lemma_well_posedness_iteration_process}.\\
    \textit{Step 1.} We show by induction that the iterative solutions resulting from $(T^{0}, \rho^{0}) = (T_\infty, \rho_*)$ satisfy
    \begin{equation} \label{monotony}
        T_* \leq T^{k} \leq T^{k-1} \leq T_\infty \quad \text{and} \quad \rho_\infty \leq \rho^{k} \leq \rho^{k-1} \leq \rho_*
    \end{equation}
    for all $k \in \mathbb{N}$. Let us verify $T_* \leq T^1 \leq T_\infty$ at first. Using
    \begin{equation}
        w \coloneqq (T^{1} - T_\infty)^+ = \begin{cases}
            T^1 - T_\infty & \text{if } T^1 > T_\infty\\
            0 & \text{if } T^1 \leq T_\infty
        \end{cases}
    \end{equation}
    as a test function in \eqref{weak_iteration_process_T} yields the identity
    \begin{equation}\label{energy_w1}
    \begin{gathered}
        \frac{1}{2} \frac{d}{dt} \| w \|_{2,\Omega}^2 + \frac{1}{R^{2}} \| \nabla w \|_{2,\Omega}^2 + \frac{L}{R} \| w \|_{2,\Gamma}^2 = \frac{\dot{R}}{R} \int_{\Omega} (\nabla w \cdot x) w \, dx
    \end{gathered}
    \end{equation}
    taking into account $J^0 = 0$, $\nabla \cdot v = 0$ in $\Omega$, and $v \cdot n = 0$ on $\Gamma$. From Young's inequality we infer
    \begin{equation}\label{Young}
        \frac{\dot{R}}{R} \int_{\Omega} (\nabla w \cdot x) w \, dx \leq \varepsilon \| \nabla w \|_{2,\Omega}^2 + \frac{C}{\varepsilon} \| w \|_{2,\Omega}^2
    \end{equation}
    due to the term $|\dot{R}| / R$ and the domain $\Omega$ being bounded. Inserting \eqref{Young} with $\varepsilon = 1 / (2 R^2)$ into \eqref{energy_w1} allows us to absorb the $L^2$-norm of the gradient. Applying Gr\"onwall's lemma to the resulting inequality
    \begin{equation}
        \frac{d}{dt} \| w \|_{2,\Omega}^2 \leq C \| w \|_{2,\Omega}^2
    \end{equation}
    implies $w = 0$ and thus $T^1 \leq T_\infty$ taking into account that $w (0) = (T_0 - T_\infty)^+ = 0$. On the other hand, testing \eqref{weak_iteration_process_T} with $\tilde{w} \coloneqq (T_* - T^1)^+$ as defined above implies
    \begin{equation}
    \begin{gathered}
        \frac{1}{2} \frac{d}{dt} \| \tilde{w} \|_{2,\Omega}^2 + \frac{1}{R^{2}} \| \nabla \tilde{w} \|_{2,\Omega}^2 + \frac{L}{R} \| \tilde{w} \|_{2,\Gamma}^2 + \frac{L}{R} \int_{\Gamma} (T_\infty - T_*) \tilde{w} \, d\sigma = \frac{\dot{R}}{R} \int_{\Omega} (\nabla \tilde{w} \cdot x) \tilde{w} \, dx
    \end{gathered}
    \end{equation}
    which allows us to conclude $T_* \leq T^1$ as before. It should be mentioned that $\tilde{w}$ is an admissible test function since $T_* \leq T_\infty$ guarantees $\tilde{w} = 0$ on $\Gamma_\infty$. The inequalities $\rho_\infty \leq \rho^{1} \leq \rho_*$ for the vapor mass density are obtained analogously.\\
    Now let us assume that \eqref{monotony} already holds for some $k \in \mathbb{N}$. Subtracting the weak formulations for $T^{k+1}$ and $T^k$ after inserting $w^{k+1} \coloneqq (T^{k+1} - T^{k})^+$ as a test function leads to the identity
    \begin{multline}\label{energy_wk}
        \frac{1}{2} \frac{d}{dt} \| w^{k+1} \|_{2,\Omega}^2 + \frac{1}{R^{2}} \| \nabla w^{k+1} \|_{2,\Omega}^2 + \frac{L}{R} \| w^{k+1} \|_{2,\Gamma}^2 = \frac{\dot{R}}{R} \int_{\Omega} (\nabla w^{k+1} \cdot x) w^{k+1} \, dx\\
        + \int_{\Gamma} \big[ J^{k-1} - J^k - L (T^{k-1} - T^{k}) \big] w^{k+1} \, d\sigma
    \end{multline}
    for all $k \geq 2$. Since our induction hypothesis \eqref{monotony} implies
    \begin{equation}
        J^{k-1} - J^k \leq \rho_\textit{sat} (T^{k-1}) - \rho_\textit{sat} (T^{k}) \leq L (T^{k-1} - T^{k}),
    \end{equation}
    we obtain $w^{k+1} = 0$ and thus $T^{k+1} \leq T^{k}$ from the same Gr\"onwall argument as before. Likewise, testing \eqref{weak_iteration_process_T} with $\tilde{w}^{k+1} \coloneqq (T_* - T^{k+1})^+$ yields $T_* \leq T^{k+1}$ taking into account that
    \begin{equation}
        \rho_\textit{sat} (T^{k}) - \rho^{k} \leq \rho_\textit{sat} (T^{k}) - \rho_\infty \leq L (T^{k} - T_*)
    \end{equation}
    already holds. The remaining inequalities $\rho_\infty \leq \rho^{k+1} \leq \rho^{k}$ for the vapor mass density are obtained analogously. This concludes the proof of \eqref{monotony} for all $k \in \mathbb{N}$ by induction.\\
    \textit{Step 2.} We are now in the position to show that the iteration process $\{ (T^k, \rho^k) \}_{k \in \mathbb{N}}$ from Lemma \ref{lemma_well_posedness_iteration_process} converges (up to a subsequence) to the desired weak solution $(T, \rho)$ of the decoupled evaporation problem. From \eqref{monotony} we infer that the pointwise limits
    \begin{equation}
        T (t,x) \coloneqq \lim_{k \to \infty} T^k (t,x) \quad \text{and} \quad \rho (t,x) \coloneqq \lim_{k \to \infty} \rho^k (t,x)
    \end{equation}
    exist for a.e. $(t,x) \in I \times \Omega$. In fact, we have $T^k \to T$ and $\rho^k \to \rho$ in $L^2 (I \times \Omega)$ and $L^2 (I \times \Gamma)$ due to Lebesgue's dominated convergence theorem. According to the energy estimates from Lemma \ref{lemma_well_posedness_iteration_process}, it follows from \eqref{monotony} that the iterative solutions $\{T^k\}_{k\in\mathbb{N}}$ and $\{\rho^k\}_{k\in\mathbb{N}}$ are both bounded in $L^2 (I, H^1(\Omega))$. Likewise, their time derivatives $\{ \partial_t T^k \}_{k\in\mathbb{N}}$ and $\{ \partial_t \rho^k \}_{k\in\mathbb{N}}$ are both bounded in $L^2 (I, X^*)$. Hence, there exist a subsequence $\{(T^{k_l}, \rho^{k_l})\}_{l \in \mathbb{N}}$ such that
    \begin{equation}
    \begin{gathered}
        T^{k_l} \rightharpoonup T \quad \text{and} \quad \rho^{k_l} \rightharpoonup \rho \quad \text{weakly in } L^2 (I, H^1(\Omega)),\\
        \partial_t T^{k_l} \rightharpoonup \partial_t T \quad \text{and} \quad \partial_t \rho^{k_l} \rightharpoonup \partial_t \rho \quad \text{weakly in } L^2 (I, X^*).
    \end{gathered}
    \end{equation}
    Consequently, the respective limits $T$ and $\rho$ satisfy the following weak formulations
    \begin{equation}
        \langle \partial_t T, \varphi \rangle + \int_{\Omega} \frac{1}{R^2} \nabla T \cdot \nabla \varphi + \frac{1}{R} \nabla T \cdot ( v - \dot{R} x ) \varphi \, dx + \frac{1}{R} \int_{\Gamma} L T \varphi \, d\sigma = \frac{1}{R} \int_{\Gamma} \big[ L T - J(T, \rho) \big] \varphi \, d\sigma,
    \end{equation}
    \begin{equation}
    \begin{gathered}
        \langle \partial_t \rho, \varphi \rangle + \int_{\Omega} \frac{1}{R^2} \nabla \rho \cdot \nabla \varphi + \frac{1}{R} \nabla \rho \cdot ( v - \dot{R} x ) \varphi \, dx + \frac{1}{R} \int_{\Gamma} \rho \varphi \, d\sigma = \frac{1}{R} \int_{\Gamma} \big[ \rho + J(T, \rho) \big] \varphi \, d\sigma
    \end{gathered}
    \end{equation}
    for a.e. $t \in I$ and all test functions $\varphi \in X$. Subtracting the redundant boundary integrals yields the desired weak formulation \eqref{Eq:1.2} and \eqref{Eq:1.1} of the decoupled evaporation problem.\\
    \textit{Step 3.} To show uniqueness, let $(T_1, \rho_1)$ and $(T_2, \rho_2)$ be two weak solutions of the decoupled droplet evaporation problem \eqref{Eq:1.2}--\eqref{drying_conditions}. Then, their differences $\delta T \coloneqq T_1 - T_2$ and $\delta \rho \coloneqq \rho_1 - \rho_2$ satisfy
    \begin{equation} \label{weak_delta_T}
        \frac{1}{2} \frac{d}{dt} \| \delta T \|_{2,\Omega}^2 + \frac{1}{R^2} \| \nabla \delta T \|_{2,\Omega}^2 = \frac{\dot{R}}{R} \int_\Omega (\nabla \delta T \cdot x) \delta T \, dx - \frac{1}{R} \int_\Gamma (\delta \rho_\textit{sat} - \delta \rho) \delta T \, d\sigma,
    \end{equation}
    \begin{equation} \label{weak_delta_rho}
        \frac{1}{2} \frac{d}{dt} \| \delta \rho \|_{2,\Omega}^2 + \frac{1}{R^2} \| \nabla \delta \rho \|_{2,\Omega}^2 = \frac{\dot{R}}{R} \int_\Omega (\nabla \delta \rho \cdot x) \delta \rho \, dx + \frac{1}{R} \int_\Gamma (\delta \rho_\textit{sat} - \delta \rho) \delta \rho \, d\sigma
    \end{equation}
    with $\delta \rho_\textit{sat} \coloneqq \rho_\textit{sat} (T_1) - \rho_\textit{sat} (T_2)$. Applying Young's inequality with $\varepsilon > 0$ yields
    \begin{equation}
        \frac{\dot{R}}{R} \int_{\Omega} (\nabla \delta T \cdot x) \delta T \, dx \leq \varepsilon \| \nabla \delta T \|_{2,\Omega}^2 + \frac{C}{\varepsilon} \| \delta T \|_{2,\Omega}^2,
    \end{equation}
    \begin{equation}
        \frac{\dot{R}}{R} \int_{\Omega} (\nabla \delta \rho \cdot x) \delta \rho \, dx \leq \varepsilon \| \nabla \delta \rho \|_{2,\Omega}^2 + \frac{C}{\varepsilon} \| \delta \rho \|_{2,\Omega}^2
    \end{equation}
    which allows us to absorb the $L^2$-norms of the gradients. Therefore, adding \eqref{weak_delta_T} and \eqref{weak_delta_rho} implies
    \begin{multline} \label{sum_of_differences}
        \frac{d}{dt} \big( \| \delta T \|_{2,\Omega}^2 + \| \delta \rho \|_{2,\Omega}^2 \big) + \| \nabla \delta T \|_{2,\Omega}^2 + \| \nabla \delta \rho \|_{2,\Omega}^2
        \leq C \big( \| \delta T \|_{2,\Omega}^2 + \| \delta \rho \|_{2,\Omega}^2 \big)\\
        - \frac{1}{R} \int_\Gamma (\delta \rho_\textit{sat} - \delta \rho) (\delta T - \delta \rho) \, d\sigma
    \end{multline}
    since $R$ is bounded away from zero. Taking into account that $\rho_\textit{sat}$ is Lipschitz continuous, we obtain
    \begin{equation} \label{boundary_Lipschitz_estimate}
    \begin{gathered}
        - \frac{1}{R} \int_\Gamma (\delta \rho_\textit{sat} - \delta \rho) (\delta T - \delta \rho) \, d\sigma \leq C \big( \| \delta T \|_{2,\Gamma}^2 + \| \delta \rho \|_{2,\Gamma}^2 \big)\\
        \leq \varepsilon \big( \| \nabla \delta T \|_{2,\Omega}^2 + \| \nabla \delta \rho \|_{2,\Omega}^2 \big) + \frac{C}{\varepsilon} \big( \| \delta T \|_{2,\Omega}^2 + \| \delta \rho \|_{2,\Omega}^2 \big)
    \end{gathered}
    \end{equation}
    using the interpolation inequality from Lemma \ref{Lemma:1}. Inserting \eqref{boundary_Lipschitz_estimate} with $\varepsilon = 1/2$ into \eqref{sum_of_differences} allows us to absorb the $L^2$-norms of the gradients. Applying Gr\"onwall's lemma to the resulting inequality
    \begin{equation}
        \frac{d}{dt} \big( \| \delta T \|_{2,\Omega}^2 + \| \delta \rho \|_{2,\Omega}^2 \big) \leq C \big( \| \delta T \|_{2,\Omega}^2 + \| \delta \rho \|_{2,\Omega}^2 \big)
    \end{equation}
    finally implies $\| \delta T(t)\|_{2,\Omega}^2 + \| \delta \rho(t) \|_{2,\Omega}^2 = 0$ for all $t \in I$ and thus $T_1 = T_2$ and $\rho_1 = \rho_2$ a.e. in $I \times \Omega$.
\end{proof}

\subsection{Short time existence for the coupled evaporation problem} \label{sec:short_time_existence}

We are now in the position to show that the coupled droplet evaporation problem from Section \ref{sec_model} admits a unique weak solution in the sense of Definition \ref{Def:weak_solution} if the time interval $I = [0, t_*)$ is sufficiently small. From now on, the droplet radius $R$ is no longer given but considered as an additional unknown. First of all, we need to understand how the weak solution of the decoupled problem from Proposition \ref{existence_Rd_given} depends on the corresponding droplet radius:

\begin{proposition} \label{H1_stability}
    Let the initial values $T_0, \rho_0 \in L^2(\Omega)$ be bounded by \eqref{initial_lower_and_upper_bounds} a.e. in $\Omega$. Further assume that the droplet radii $R_1, R_2 \in \Sigma_{t_*}$ are bounded away from zero in the sense that $R_1 (t) \geq \eta_1 > 0$ and $R_2 (t) \geq \eta_2 > 0$ for a.e. $t \in I$. Then, the corresponding weak solutions $(T_1, \rho_1)$ and $(T_2, \rho_2)$ of the decoupled droplet evaporation problem \eqref{Eq:1.2}--\eqref{drying_conditions} satisfy the following stability estimates
    \begin{equation} \label{Hinf_stability_estimate}
    \begin{gathered}
        \sup_{t \in I} \| T_1 - T_2 \|_{2,\Omega}^2 + \| \nabla (T_1 - T_2) \|_{2, I\times\Omega}^2 \leq C \| R_1 - R_2 \|_{H^1(I)}^2,\\
        \sup_{t \in I} \| \rho_1 - \rho_2 \|_{2,\Omega}^2 + \| \nabla (\rho_1 - \rho_2) \|_{2, I\times\Omega}^2 \leq C \| R_1 - R_2 \|_{H^1(I)}^2
    \end{gathered}
    \end{equation}
    where the Lipschitz constant $C > 0$ depends on the lower bound $\eta \coloneqq \min \{ \eta_1, \eta_2 \}$ of both droplet radii.
\end{proposition}

\begin{proof}
    Let us first briefly outline the main idea of our proof: After subtracting the weak formulations of both solutions, we sort all defect terms to the right-hand side and estimate them with respect to the difference of both droplet radii. Applying Gr\"onwall's lemma to the resulting inequality then finally implies the desired stability estimate.\\
    Since $R_1, R_2 \in \Sigma_{t_*}$ are bounded away from zero, the corresponding weak solutions $(T_1, \rho_1)$ and $(T_2, \rho_2)$ from Proposition \ref{existence_Rd_given} are well-defined.
    Testing \eqref{Eq:1.2} with the difference $\varphi = T_1 - T_2 \eqqcolon \delta T$ yields two equations for $T_1$ and $T_2$, respectively. Subtracting one from the other implies
    \begin{multline} \label{diff_wf_T12}
        \langle \partial_t \delta T, \delta T \rangle + \int_{\Omega} \nabla \left( \frac{T_1}{R_1^{2}} - \frac{T_2}{R_2^{2}} \right) \cdot \nabla \delta T \, dx
        + \int_{\Omega} \nabla \cdot \left( \frac{T_1 v_1}{R_1} - \frac{T_2 v_2}{R_2} \right) \delta T \, dx\\
        - \int_{\Omega} \nabla \left( \frac{\dot{R}_1 T_1}{R_1} - \frac{\dot{R}_2 T_2}{R_2} \right) \cdot x \delta T \, dx + \int_{\Gamma} \left( \frac{J_1}{R_1} -  \frac{J_2}{R_2} \right) \delta T \, d\sigma = 0
    \end{multline}
    with the evaporation rates $J_i \coloneqq J (T_i, \rho_i)$ for $i \in \{1,2\}$. We define $\delta R \coloneqq R_1 - R_2$, $\delta \dot{R} \coloneqq \dot{R}_1 - \dot{R}_2$, $\delta v \coloneqq v_1 - v_2$, and $\delta J \coloneqq J_1 - J_2$ for the ease of presentation. Inserting the following decompositions
    \begin{equation} \label{first_decomposition}
        \frac{T_1}{R_1^2} - \frac{T_2}{R_2^2} = \frac{\delta T}{R_1^2} - \delta R \frac{R_1 + R_2}{R_1 R_2} T_2,
    \end{equation}
    \begin{equation}
        \frac{T_1 v_1}{R_1} - \frac{T_2 v_2}{R_2} = \frac{\delta T v_1}{R_1} + \left( \frac{\delta v}{R_1} - \frac{\delta R v_2}{R_1 R_2} \right) T_2,
    \end{equation}
    \begin{equation}
        \frac{\dot{R}_1 T_1}{R_1} - \frac{\dot{R}_2 T_2}{R_2} = \frac{\dot{R}_1}{R_1} \delta T + \left( \frac{\delta \dot{R}}{R_1} - \frac{\delta R \dot{R}_2}{R_1 R_2} \right) T_2,
    \end{equation}
    \begin{equation} \label{last_decomposition}
        \frac{J_1}{R_1} -  \frac{J_2}{R_2} = \frac{\delta J}{R_1} -  \frac{\delta R J_2}{R_1 R_2}
    \end{equation}
    into \eqref{diff_wf_T12} and sorting the resulting terms yields
    \begin{equation}\label{energy_est_Tg}
        \frac{1}{2} \frac{d}{dt} \| \delta T \|_{2,\Omega}^2 + \frac{1}{R_1^2} \| \nabla \delta T \|_{2,\Omega}^2 = \frac{\dot{R}_1}{R_1} \int_{\Omega} (\nabla \delta T \cdot x) \delta T \, dx - \frac{1}{R_1} \int_{\Gamma} \delta J \delta T \, d\sigma + d_T
    \end{equation}
    where the temperature defect
    \begin{multline} \label{defect_T}
        d_T \coloneqq \frac{1}{R_1 R_2} \int_{\Omega} \nabla T_2 \cdot \big( \delta R v_2 - R_2 \delta v + \delta \dot{R} R_2 x - \delta R \dot{R}_2 x \big) \delta T \, dx + \frac{\delta R}{R_1 R_2} \int_{\Gamma} J_2 \delta T \, d\sigma \\+ \delta R \frac{R_1 + R_2}{R_1 R_2} \int_{\Omega} \nabla T_2 \cdot \nabla \delta T \, dx
    \end{multline}
    contains all products with $\delta R$, $\delta \dot{R}$, or $\delta v$. Notice that the integral of $(\nabla \delta T \cdot v_2) \delta T$ over~$\Omega$ vanishes since $v_2$ is divergence-free and tangential at the droplet surface. Young's inequality implies
    \begin{equation} \label{Young_grad}
        \frac{\dot{R}_1}{R_1} \int_{\Omega} (\nabla \delta T \cdot x) \delta T \, dx \leq \frac{1}{2 R_1^2} \| \nabla \delta T \|_{2,\Omega}^2 + C \| \delta T \|_{2,\Omega}^2
    \end{equation}
    taking into account $|\dot{R}_1| \leq J_\infty$ and $|\nabla \delta T \cdot x| \leq C |\nabla \delta T|$ due to $\Omega$ being bounded. Inserting \eqref{Young_grad} into \eqref{energy_est_Tg} allows us to absorb the $L^2$-norm of the gradient. By that, we obtain
    \begin{equation} \label{diff_T}
        \frac{d}{dt} \| \delta T \|_{2,\Omega}^2 + \frac{1}{R_1^2} \| \nabla \delta T \|_{2,\Omega}^2
        \leq C \| \delta T \|_{2,\Omega}^2 - \frac{2}{R_1} \int_{\Gamma} \delta J \delta T \, d\sigma + 2 d_T
    \end{equation}
    for a.e. $t \in I$. 
    
    Likewise, testing \eqref{Eq:1.1} with $\varphi = \rho_1 - \rho_2 \eqqcolon \delta \rho$ yields
    \begin{equation} \label{diff_rho}
        \frac{d}{dt} \| \delta \rho \|_{2,\Omega}^2 + \frac{1}{R_1^2} \| \nabla \delta \rho \|_{2,\Omega}^2
        \leq C \| \delta \rho \|_{2,\Omega}^2 + \frac{2}{R_1} \int_{\Gamma} \delta J \delta \rho \, d\sigma + 2 d_\rho
    \end{equation}
    with the vapor mass defect $d_\rho$ being defined by
    \begin{multline} \label{defect_rho}
        d_\rho \coloneqq \frac{1}{R_1 R_2} \int_{\Omega} \nabla \rho_2 \cdot \big( \delta R v_2 - R_2 \delta v + \delta \dot{R} R_2 x - \delta R \dot{R}_2 x \big) \delta \rho \, dx - \frac{\delta R}{R_1 R_2} \int_{\Gamma} J_2 \delta \rho \, d\sigma\\
        + \delta R \frac{R_1 + R_2}{R_1 R_2} \int_{\Omega} \nabla \rho_2 \cdot \nabla \delta \rho \, dx
    \end{multline}
    applying the same decompositions \eqref{first_decomposition}--\eqref{last_decomposition} as before. Adding \eqref{diff_T} and \eqref{diff_rho} implies
    \begin{multline} \label{energy_estimate}
        \frac{d}{dt} \big( \| \delta T \|_{2,\Omega}^2 + \| \delta \rho \|_{2,\Omega}^2 \big) + \frac{1}{R_1^2} \big( \| \nabla \delta T \|_{2,\Omega}^2 + \| \nabla \delta \rho \|_{2,\Omega}^2 \big)
        \leq C \big( \| \delta T \|_{2,\Omega}^2 + \| \delta \rho \|_{2,\Omega}^2 \big)\\
        + \frac{2}{R_1} \int_{\Gamma} |\delta J | |\delta T - \delta \rho| \, d\sigma + d
    \end{multline}
    where $d \coloneqq 2 (d_T + d_\rho)$ denotes the total defect. The rough estimate
    \begin{equation}
        2 \int_{\Gamma} |\delta J | |\delta T - \delta \rho| \, d\sigma \leq C \big( \| \delta T \|_{2,\Gamma}^2 + \| \delta \rho \|_{2,\Gamma}^2 \big)
    \end{equation}
    follows from $\rho_\textit{sat}$ being Lipschitz continuous. Together with the interpolation inequality \eqref{Ehlring_inequality} we obtain
    \begin{equation} \label{Jw_estimate}
        \frac{2}{R_1} \int_{\Gamma} |\delta J | |\delta T - \delta \rho| \, d\sigma \leq \frac{1}{2 R_1^2} \left( \| \nabla \delta T \|_{2,\Omega}^2 + \| \nabla \delta \rho \|_{2,\Omega}^2 \right) + C \left( \| \delta T \|_{2,\Omega}^2 + \| \delta \rho \|_{2,\Omega}^2 \right) 
    \end{equation}
    which allows us to absorb the $L^2$-norms of the gradients. Therefore, inserting \eqref{Jw_estimate} into \eqref{energy_estimate} implies
    \begin{equation}\label{energy_estimate_with_defect}
        \frac{d}{dt} \big( \| \delta T \|_{2,\Omega}^2 + \| \delta \rho \|_{2,\Omega}^2 \big) + \frac{1}{2 R_1^2} \big( \| \nabla \delta T \|_{2,\Omega}^2 + \| \nabla \delta \rho \|_{2,\Omega}^2 \big) \leq C \big( \| \delta T \|_{2,\Omega}^2 + \| \delta \rho \|_{2,\Omega}^2 \big) + d
    \end{equation}
    for a.e. $t \in I$. As carried out in the Appendix, the total defect is bounded by
    \begin{equation} \label{defect_estimate}
        |d| \leq \frac{1}{4 R_1^2} \big( \| \nabla \delta T \|_{2, \Omega}^2 + \| \nabla \delta \rho \|_{2, \Omega}^2 \big) + C \Big[ \| \delta T \|_{2, \Omega}^2 + \| \delta \rho \|_{2, \Omega}^2 + \big( 1 + \| \nabla T_2 \|_{2, \Omega}^2 + \| \nabla \rho_2 \|_{2, \Omega}^2 \big) \delta R^2 + \delta \dot{R}^2 \Big]
    \end{equation}
    which directly follows from the repeated application of H\"older's and Young's inequalities together with the interpolation inequality from Lemma \ref{Lemma:1}. After combining \eqref{energy_estimate_with_defect} and \eqref{defect_estimate}, we absorb the $L^2$-norms of the gradients. By that, we obtain
    \begin{multline}
        \frac{d}{dt} \big( \| \delta T \|_{2,\Omega}^2 + \| \delta \rho \|_{2,\Omega}^2 \big) + \| \nabla \delta T \|_{2,\Omega}^2 + \| \nabla \delta \rho \|_{2,\Omega}^2 \leq C \Big[ \| \delta T \|_{2, \Omega}^2 + \| \delta \rho \|_{2, \Omega}^2\\
        + \left( 1 + \| \nabla T_2 \|_{2, \Omega}^2 + \| \nabla \rho_2 \|_{2, \Omega}^2 \right) \delta R^2 + \delta \dot{R}^2 \Big]
    \end{multline}
    since $R_1 \geq \eta_1$ is bounded away from zero. Integration over time implies
    \begin{equation}
        E(t) \leq C \int_0^t E(\tau) \, d\tau + D(t)
    \end{equation}
    with the energy
    \begin{equation}
        E (t) \coloneqq \| \delta T (t) \|_{2,\Omega}^2 + \| \delta \rho (t) \|_{2,\Omega}^2 + \int_0^t \| \nabla \delta T \|_{2,\Omega}^2 + \| \nabla \delta \rho \|_{2,\Omega}^2 \, d\tau
    \end{equation}
    and the integrated total defect
    \begin{equation}\label{integrated_total_defect}
        D(t) \coloneqq C \int_0^t \left( 1 + \| \nabla T_2 \|_{2, \Omega}^2 + \| \nabla \rho_2 \|_{2, \Omega}^2 \right) \delta R^2 + \delta \dot{R}^2 \, d\tau
    \end{equation}
    for $t \in I$.
    We are now in the position to apply Gr\"onwall's lemma. The latter implies
    \begin{equation} \label{Gronwall}
        E(t) \leq \exp \left( C t \right) D(t)
    \end{equation}
    since $D$ is non-decreasing. From H\"older's inequality we infer the $L^\infty$-estimate
    \begin{equation} \label{estimate_sup_delta_R}
        \sup_{t \in I} \delta R^2 \leq \left( \int_I |\delta \dot{R}| \, d\tau \right)^2 \leq t_* \| \delta \dot{R} \|_{2,I}^2 \leq t_* \| \delta R \|_{H^1(I)}^2
    \end{equation}
    since $\delta R (0) = 0$ by definition. Inserting \eqref{estimate_sup_delta_R} into \eqref{integrated_total_defect} yields the desired upper bound
    \begin{equation} \label{first_defect_estimate}
        D(t) \leq C \big[ 1 + t_* \left( t + \| \nabla T_2 \|_{2, I\times\Omega}^2 + \| \nabla \rho_2 \|_{2, I\times\Omega}^2 \right) \big] \| \delta R \|_{H^1(I)}^2 \leq C \| \delta R \|_{H^1(I)}^2
    \end{equation}
    for the integrated total defect. Therefore, combining \eqref{Gronwall} and \eqref{first_defect_estimate} concludes our proof.
\end{proof}

Recall from the beginning of this section that finding a weak solution in the sense of Definition \ref{Def:weak_solution} can be reformulated as a fixed-point problem. In the following, we finally apply Banach's fixed-point theorem to prove the short time existence of a unique weak solution to our droplet evaporation problem.

\begin{theorem} \label{theorem_short_time_existence}
    Let the initial values $T_0, \rho_0 \in L^2(\Omega)$ be bounded by \eqref{initial_lower_and_upper_bounds} a.e. in $\Omega$. Then, the Volterra operator $\mathcal{T}: \Sigma_{t_*} \to \Sigma_{t_*}$ defined by \eqref{volterra} is a contraction with respect to the $H^1$-norm if the underlying time interval $I = [0, t_*)$ is sufficiently small. Consequently, the droplet evaporation problem from Section~\ref{sec_model} admits a unique weak solution $(R, T, \rho)$ on $I$ in the sense of Definition \ref{Def:weak_solution}.
\end{theorem}

\begin{proof}
    Assume $t_* < R_0 / (2 J_\infty)$ such that the droplet radii in $\Sigma_{t_*}$ are uniformly bounded away from zero. Now let $R_1, R_2 \in \Sigma_{t_*}$ and define $\delta \mathcal{T} \coloneqq \mathcal{T}(R_1) - \mathcal{T}(R_2)$. From $\delta \mathcal{T} (0) = 0$ we infer
        \begin{equation}
            \| \delta \mathcal{T} \|^2_{H^1(I)} = \| \delta \mathcal{T} \|^2_{2,I} + \| \delta \dot{\mathcal{T}} \|^2_{2,I} \leq \left( 1 + \int_I t \, dt \right) \| \delta \dot{\mathcal{T}} \|_{2,I}^2 = \left( 1 + \frac{t_*^2}{2} \right) \| \delta \dot{\mathcal{T}} \|_{2,I}^2
        \end{equation}
    taking into account $|\delta \mathcal{T} (t)| \leq \sqrt{t} \| \delta \dot{\mathcal{T}} \|_{2,I}$ due to H\"older's inequality. The latter further implies
    \begin{equation}
        \| \delta \dot{\mathcal{T}} \|_{2,I}^2 = 
        \int_I \left( \frac{1}{4 \pi} \int_\Gamma \delta J \, d\sigma \right)^2 d\tau
        \leq \frac{1}{4 \pi} \int_I \int_\Gamma \delta J^2 \, d\sigma d\tau
    \end{equation}
    since $|\Gamma| = 4 \pi$. Recall that the rough estimate $\delta J^2 \leq C (\delta T^2 + \delta \rho^2)$ follows from Young's inequality and~$\rho_\textit{sat}$ being Lipschitz continuous. Together with Lemma \ref{Lemma:1} we thus obtain
    \begin{equation}
    \begin{gathered}
        \| \delta \dot{\mathcal{T}} \|_{2,I}^2 \leq C \int_I \varepsilon \big( \| \nabla \delta T \|_{2, \Omega}^2 + \| \nabla \delta \rho \|_{2, \Omega}^2 \big) + \frac{1}{\varepsilon} \big( \| \delta T \|_{2, \Omega}^2 + \| \delta \rho \|_{2, \Omega}^2 \big) \,d\tau\\
        \leq C \bigg[ \varepsilon \big( \| \nabla \delta T \|_{2, I\times\Omega}^2 + \| \nabla \delta \rho \|_{2, I\times\Omega}^2 \big) + \frac{t_*}{\varepsilon} \sup_{t \in I} \big( \| \delta T \|_{2, \Omega}^2 + \| \delta \rho \|_{2, \Omega}^2 \big) \bigg]
    \end{gathered}
    \end{equation}
    which finally makes Proposition \ref{H1_stability} applicable. By that, we arrive at the desired inequality
    \begin{equation}
        \| \delta \mathcal{T} \|^2_{H^1(I)} \leq C \bigg( \varepsilon + \frac{t_*}{\varepsilon} \bigg) \| \delta R \|_{H^1(I)}^2
    \end{equation}
    whose Lipschitz constant becomes strictly smaller than one if $\varepsilon$ and $t_*$ are chosen accordingly. For instance, $\varepsilon = 1 / (2 C)$ and $t_* < 1 / (2 C)^2$ would be one possible choice. Hence, the operator $\mathcal{T}$ turns out to be a contraction on $\Sigma_{t_*}$ if $t_* > 0$ is sufficiently small. The short time existence of a unique weak solution $(R, T, \rho)$ then follows from Banach's fixed-point theorem.
\end{proof}

\subsection{Maximal time interval of existence} \label{sec_maximal_time_interval}

The weak solution $(R, T, \rho)$ of the coupled evaporation problem is said to be maximal if it cannot be extended on a larger time interval. The following theorem shows that the maximal time interval can only be finite if the droplet evaporates completely.

\begin{theorem} \label{theorem_maximal_existence}
Let the initial values $T_0, \rho_0 \in L^2(\Omega)$ be bounded by \eqref{initial_lower_and_upper_bounds} a.e. in $\Omega$. Then, the droplet evaporation problem admits a unique weak solution $(R, T, \rho)$ in the sense of Definition \ref{Def:weak_solution} whose time interval $I_\textit{max} \coloneqq [0, t_\textit{max})$ is maximal. If $t_\textit{max} < \infty$, then the droplet evaporates completely meaning that
\begin{equation} \label{complete_evaporation}
    \lim_{t \to t^-_{max}} R (t) = 0
\end{equation}
where $t \to t^-_{max}$ denotes the one-sided limit from below.
\end{theorem}

\begin{proof}
First of all, it should be mentioned that \eqref{complete_evaporation} is well-defined since the Sobolev embedding $H^1(I_\textit{max})\xhookrightarrow{} C^0([0,t_\textit{max}])$ ensures the required continuity of the droplet radius.
Now assume that 
\begin{equation}
    R_\textit{max} \coloneqq \lim_{t \to  t^-_\textit{max}} R (t) > 0
\end{equation}
is strictly positive. According to the Lions--Magenes lemma, there exist $T_\textit{max}, \rho_\textit{max} \in L^2(\Omega)$ such that
\begin{equation}
    \lim_{t \to t^-_\textit{max}} \| T (t) - T_\textit{max} \|_{2,\Omega} = 0 \text{\qquad and \qquad} \lim_{t \to t^-_\textit{max}} \| \rho (t) - \rho_\textit{max} \|_{2,\Omega} = 0.
\end{equation}
According to Proposition \ref{existence_Rd_given}, we have $\rho_\infty \leq \rho_\textit{max} \leq \rho_*$ and $T_* \leq T_\textit{max} \leq T_\infty$ a.e. in $\Omega$. Therefore, the limits $(R_\textit{max}, T_\textit{max}, \rho_\textit{max})$ are admissible initial values for the droplet evaporation problem. Theorem \ref{theorem_short_time_existence} thus yields a unique weak solution on $[t_\textit{max}, t_\textit{max} + \varepsilon)$ for some $\varepsilon > 0$. Concatenating the weak solutions on
$I_\textit{max}$ and $[t_\textit{max}, t_\textit{max} + \varepsilon)$ defines a weak solution on the extended time interval $[0, t_\textit{max} + \varepsilon)$ which contradicts $I_\textit{max}$ being maximal. Consequently, $R_\textit{max}$ has to be zero.
\end{proof}

\section{Numerical examples} \label{sec_numerical_examples}

The following numerical examples illuminate how different air flows around a single spherical droplet affect its evaporation process. For the ease of presentation, we always consider a $1\,\mu$l water droplet evaporating at $T_\infty = 60\,$\textdegree C and $\text{RH}_\infty = 10\,$\% where $\text{RH}_\infty \coloneqq \rho_\infty / \rho_\textit{sat}(T_\infty)$ denotes the relative humidity of the drying air. Other drying conditions are studied in \cite{baensch18, doss22, doss23} with similar direct numerical simulations. From \cite{doss23} we also adopt the physical parameters of our model.\\
The considered air flows will be rotationally symmetric in vertical direction. Therefore, transforming our single droplet evaporation model into spherical coordinates allows us to neglect the azimuth. Let $\theta \in [0, \pi]$ be the polar angle and $r \geq 0$ the radial distance to the droplet center. Taking into account the aforementioned symmetry, all simulations are performed in the computational domain
\begin{equation}
    \Omega = \big\{ (\theta, r) \in \mathbb{R}^2 \, : \, 0 \leq \theta \leq \pi,\, 1 \leq r \leq 50 \big\}
\end{equation}
representing the two-dimensional generatrix of the rescaled gas phase. We use Gmsh (version 4.4.1) to generate a structured mesh for $\Omega$ consisting of 58\,882 triangles. As shown in Figure \ref{fig:mesh},
\begin{figure}
    \centering
    \includegraphics[width=0.6\textwidth]{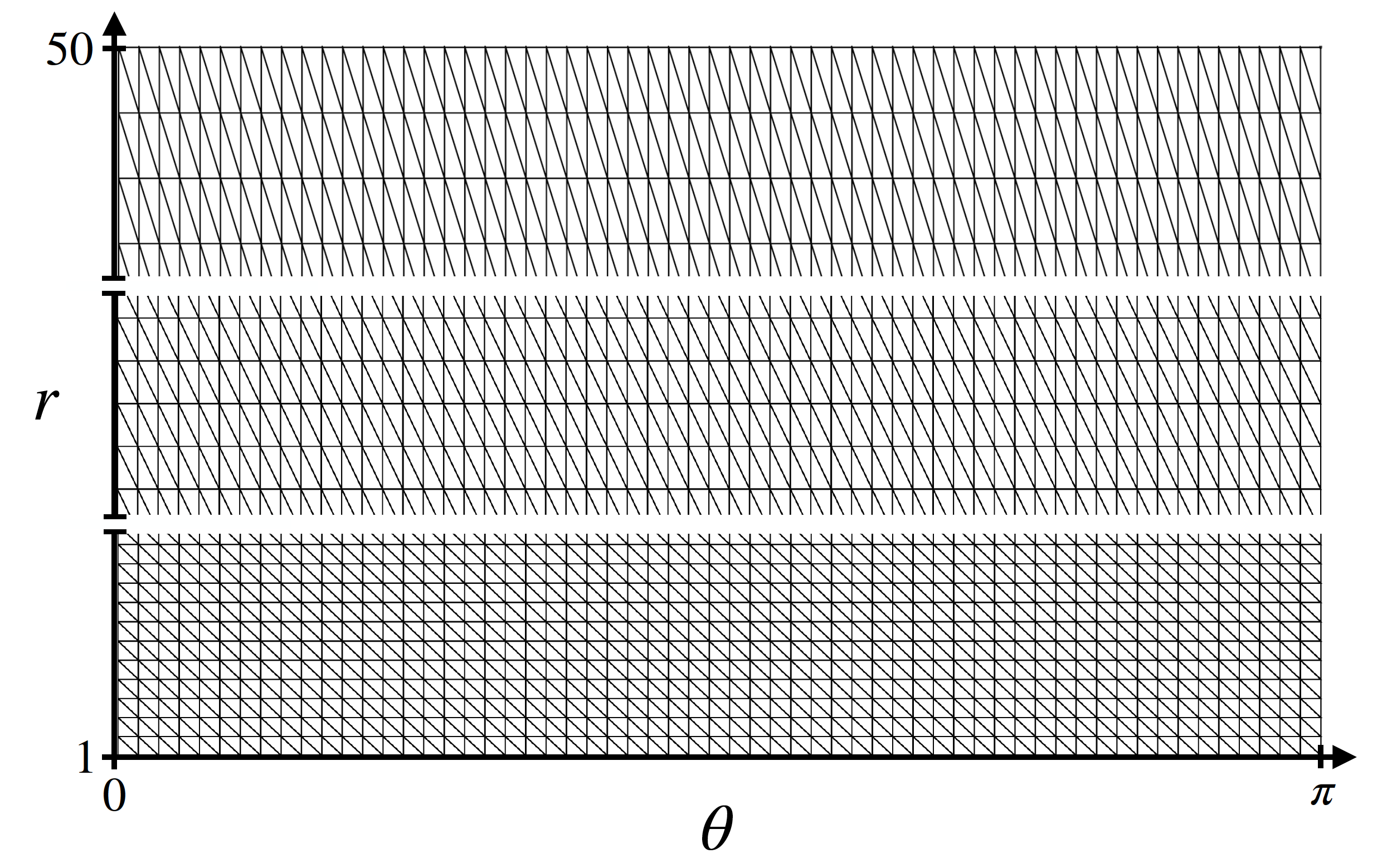}
    \caption{Mesh refinement towards the droplet. Due to the large aspect ratio of our computational domain, only three sections of the mesh are shown. The $r$-axis is interrupted accordingly.}
    \label{fig:mesh}
\end{figure}
their resolution increases towards the droplet surface. To be more specific, the triangles are successively stretched by~$0.25\,$\% in each horizontal layer starting from the droplet surface. We then apply the finite element method (FEM) using a semi-implicit Euler scheme to compute the temperature and vapor mass distributions in a monolithic manner. Apart from the droplet radius, all unknowns are treated implicitly and Newton's method is applied in each time step. Regarding the accuracy of our numerical results, the uniform time step size $\Delta t = 1\,$s was found to be sufficiently small.  Second-order Lagrange elements serve as basis functions for our FEM simulations which are implemented in \texttt{C++} with the Distributed and Unified Numerics Environment DUNE \cite{sander20, bastian21, engwer17, engwer18}.  The corresponding linear systems are assembled automatically by the module dune-fufem\footnote{Module description of dune-fufem: \url{https://www.dune-project.org/modules/dune-fufem/}} from our weak formulations after being linearized and transformed into spherical coordinates. We apply UMFPACK \cite{davis04} as a direct solver to compute the resulting $237\,762$ degrees of freedom. Our numerical results are finally visualized with ParaView (version 5.7.0) after being transformed back into Euclidean coordinates. The resulting images in Figures \ref{fig:air_flows}, \ref{fig:heat_and_mass_transfer_Stokes_flow}, and \ref{fig:heat_and_mass_transfer_acoustic_streaming} do not show the whole computational domain, but only the interesting area near the droplet.

\subsection{Stokes flow} \label{sec_Stokes_flow}

We first consider the evaporation of a spherical droplet exposed to a laminar air flow in the ambient gas phase. The latter is described by the following analytical solution of the Stokes equations \cite{landau87}
\begin{equation} \label{Stokes_flow}
\begin{aligned}
    v_\theta &= -V_\infty \sin \theta \left( 1 - \frac{1}{4 r^3} - \frac{3}{4r} \right), \\
    v_r &= V_\infty \cos \theta \left( 1 + \frac{1}{2 r^3} - \frac{3}{2r} \right)
\end{aligned}
\end{equation}
where $v_\theta$ and $v_r$ denote the velocity components in polar and radial direction, respectively. The flow parameter $V_\infty \geq 0$ determines the ambient velocity far away from the droplet. Notice that the required Lipschitz continuity \eqref{air_flow_continuity} holds since \eqref{Stokes_flow} does not even depend on the droplet radius.\\
Figure \ref{fig:Stokes_flow}
\begin{figure}
    \centering
    \begin{subfigure}[b]{0.48\textwidth}
         \centering
         \includegraphics[height=150pt]{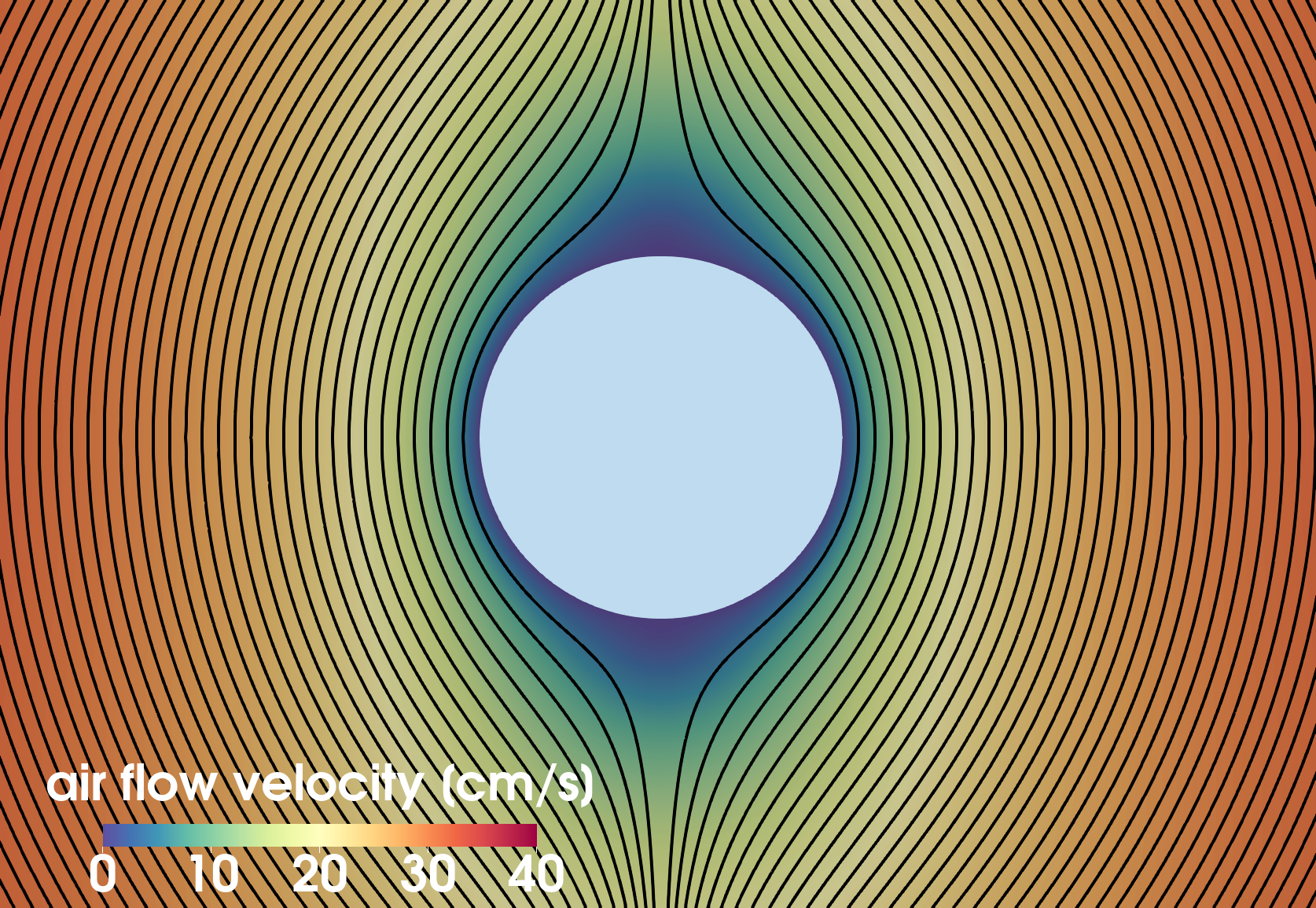}
         \caption{Stokes flow at $V_\infty = 40\,$cm/s.}
         \label{fig:Stokes_flow}
    \end{subfigure}
    \hfill
    \begin{subfigure}[b]{0.48\textwidth}
         \centering
         \includegraphics[height=150pt]{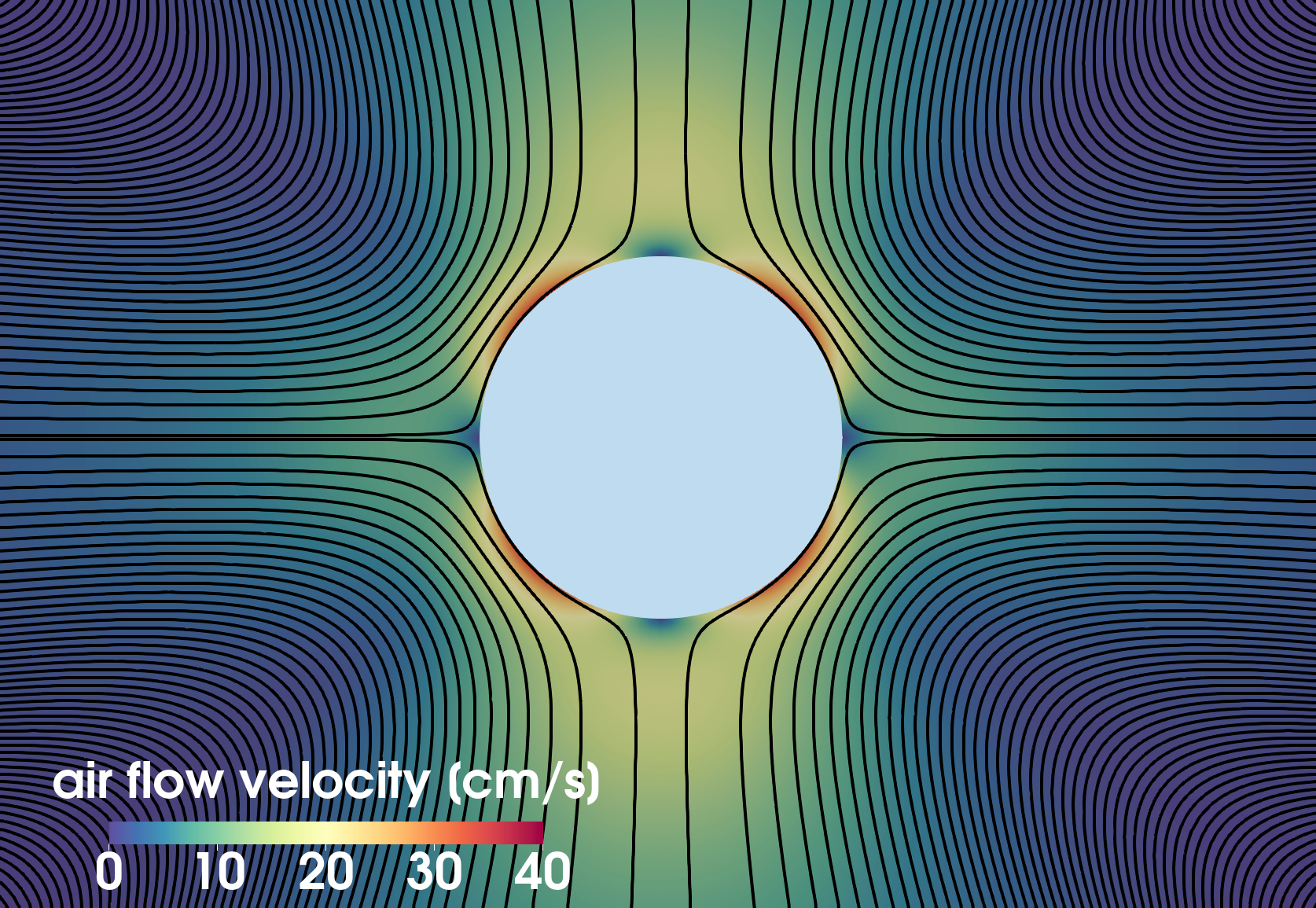}
         \caption{Acoustic streaming at $\text{SPL} = 164\,$dB.}
         \label{fig:acoustic_streaming}
    \end{subfigure}
    \caption{Air flows around the $1\,\mu$l droplet considered in our numerical examples.}
    \label{fig:air_flows}
\end{figure}
illustrates the Stokes flow around the droplet for $V_\infty = 40\,$cm/s. The convective impact of the Stokes flow on heat and vapor mass distributions is visualized in Figure \ref{fig:heat_and_mass_transfer_Stokes_flow}.
\begin{figure}
    \centering
    \begin{subfigure}[b]{0.48\textwidth}
         \centering
         \includegraphics[height=140pt]{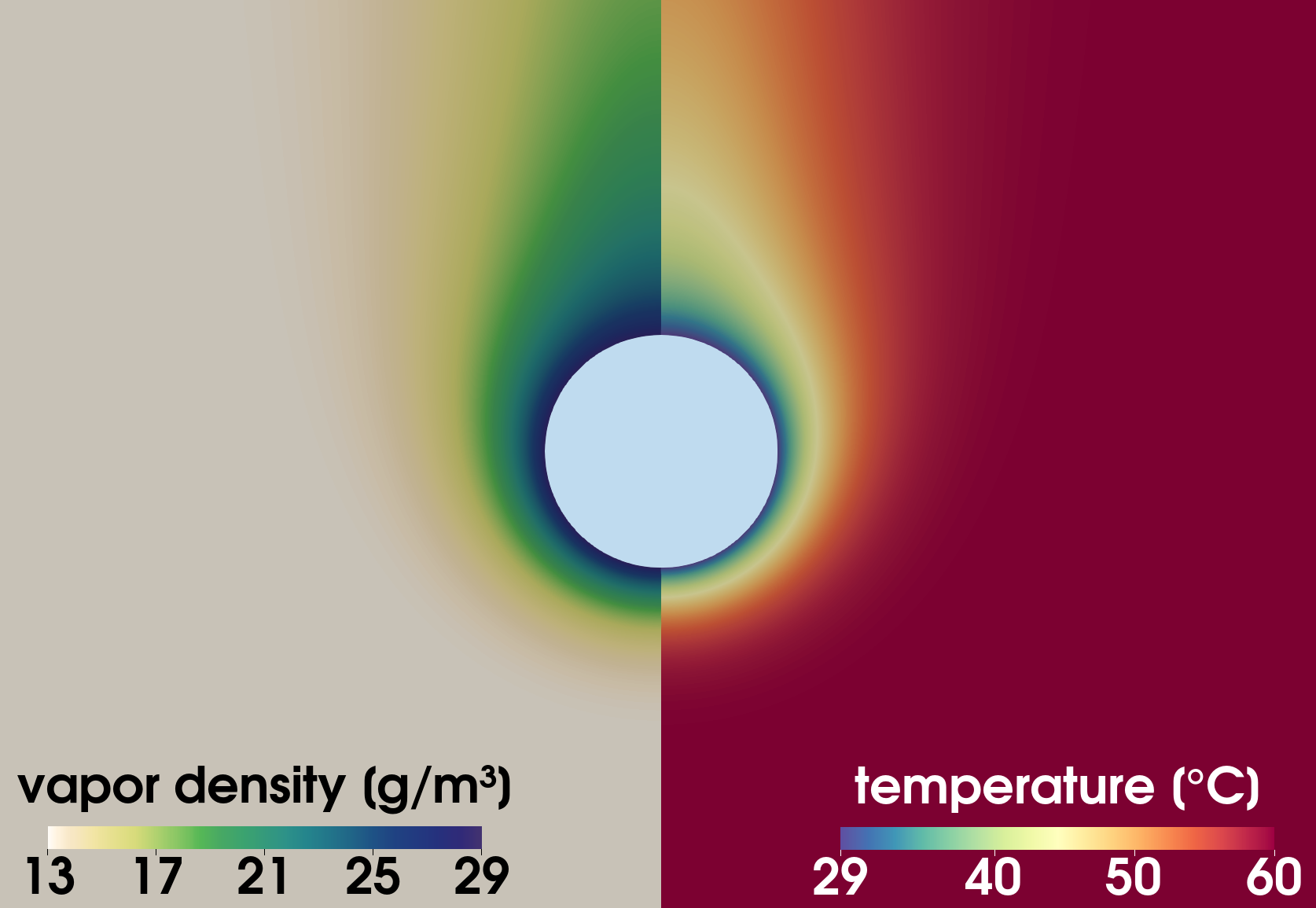}
         \caption{Heat and vapor mass transport at $V_\infty = 40\,$cm/s.}
         \label{fig:heat_and_mass_transfer_Stokes_flow}
    \end{subfigure}
    \hfill
    \begin{subfigure}[b]{0.48\textwidth}
         \centering
         \includegraphics[height=130pt]{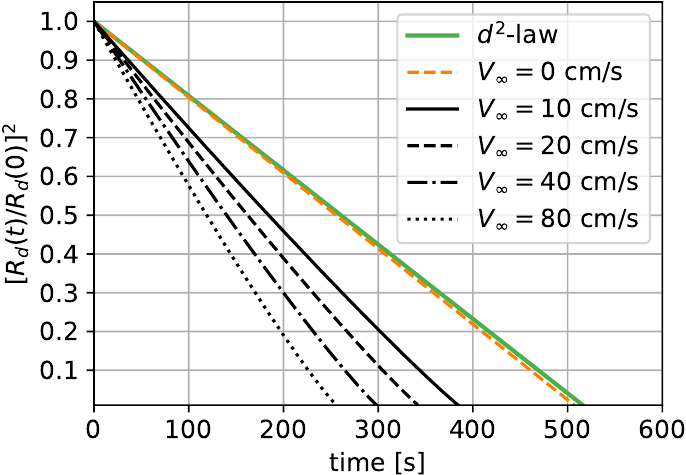}
         \caption{Evolution of the normed squared droplet radius.}
         \label{fig:droplet_radii_Stokes_flow}
    \end{subfigure}
    \caption{Convective impact of the Stokes flow on the evaporation process of a $1\,\mu$l water droplet under the drying conditions $T_\infty = 60\,$\textdegree C and $\text{RH}_\infty = 10\,$\%.}
    \label{fig:evaporation_Stokes_flow}
\end{figure}
One can clearly see that the vapor evaporating from the droplet is blown away in vertical direction. Likewise, the air above the droplet has a lower temperature than the one below or around its equator. Therefore, the evaporation rate increases from the north towards the south pole of the droplet where the normal heat flux assumes its maximum. Recall from Section \ref{sec1} that the evaporation of a spherical droplet into stagnant air is well described by the $d^2$-law. According to \cite{doss22, doss23}, the latter is given by
\begin{equation} \label{d2_law}
\frac{d R_d^2}{dt} = - \frac{2 k_g (T_{\infty} - T_d)}{\rho_w \Lambda_w}
\end{equation}
together with the following implicit equation
\begin{equation} \label{Td_estimate}
    \frac{\rho_\textit{sat}(T_d) - \rho_\infty}{T_\infty - T_d} = \frac{k_g}{D_v \Lambda_w}
\end{equation}
for the uniform temperature $T_d$ of the droplet. Figure \ref{fig:droplet_radii_Stokes_flow} compares the evolution of the normed squared droplet radius obtained from the $d^2$-law and our direct numerical simulations. If the latter are performed without convection, both curves coincide almost exactly. Notice that the simulated droplet evaporates slightly faster than predicted by the $d^2$-law. This is due to the truncation of the gas phase at the rescaled distance $r = 50$ from the droplet center. The Stokes flow, on the other hand, makes the droplet evaporate significantly faster. Depending on the ambient gas velocity, the lifetime of the droplet is reduced by up to $50\,$\%. Furthermore, the normed squared droplet radii are no longer linear but slightly convex. Therefore, the evaporation rate decreases as the droplet becomes smaller.

\subsection{Acoustic streaming}

Our second numerical example addresses the convection inside an acoustic levitator and requires a brief introduction to understand its physical background. Acoustic leviation is a very elegant way to perform single droplet drying experiments \cite{schiffter07}. Avoiding physical contact, a standing ultrasound wave is used to immobilize a single droplet in one of its sound pressure nodes. More precisely, the so-called acoustic radiation pressure counteracts gravity \cite{andrade18}. However, the rapid attenuation of the sound wave at the droplet surface also causes fluid circulations inside the viscous boundary layer. This microscopic effect induces a steady air flow in the ambient gas phase which is commonly known as outer acoustic streaming \cite{yarin99}. According to \cite{lee90}, its polar and radial components are given by
\begin{equation} \label{outer_acoustic_streaming}
\begin{aligned}
    v_\theta &= - \frac{45 A^2}{32 \omega R_d \rho_g^2 c_0^2} \frac{1}{r^4} \sin 2\theta, \\
    v_r &= \frac{45 A^2}{32 \omega R_d \rho_g^2 c_0^2} \bigg( \frac{1}{r^2} - \frac{1}{r^4} \bigg) ( 3 \cos^2 \theta - 1 )
\end{aligned}
\end{equation}
where $A$ denotes the sound pressure amplitude, $\omega$ the angular sound frequency, and $c_0$ the speed of sound in air. It should be pointed out that \eqref{outer_acoustic_streaming} only applies to a spherical droplet whose center lies in one of the sound pressure nodes generated by a planar acoustic levitation wave. The sound pressure amplitude~$A$ is usually expressed in terms of the corresponding sound pressure level (SPL) given by
\begin{equation} \label{SPL}
	\text{SPL} = \left[ 20 \cdot \log_{10} (A/\text{Pa}) + 94 \right] \, \text{dB}
\end{equation}
in decibel \cite{yarin99}. Unlike the Stokes flow \eqref{Stokes_flow} considered in Section \ref{sec_Stokes_flow}, the acoustic streaming now actually depends on the droplet radius. The Lipschitz condition \eqref{air_flow_continuity} can be easily verified.\\
Figure \ref{fig:acoustic_streaming} illustrates the acoustic streaming around a $1\,\mu$l droplet levitated at $\text{SPL} = 164\,$dB. The corresponding heat and vapor mass distributions in the ambient gas phase are visualized in Figure \ref{fig:heat_and_mass_transfer_acoustic_streaming}.
\begin{figure}
    \centering
    \begin{subfigure}[b]{0.48\textwidth}
         \centering
         \includegraphics[height=140pt]{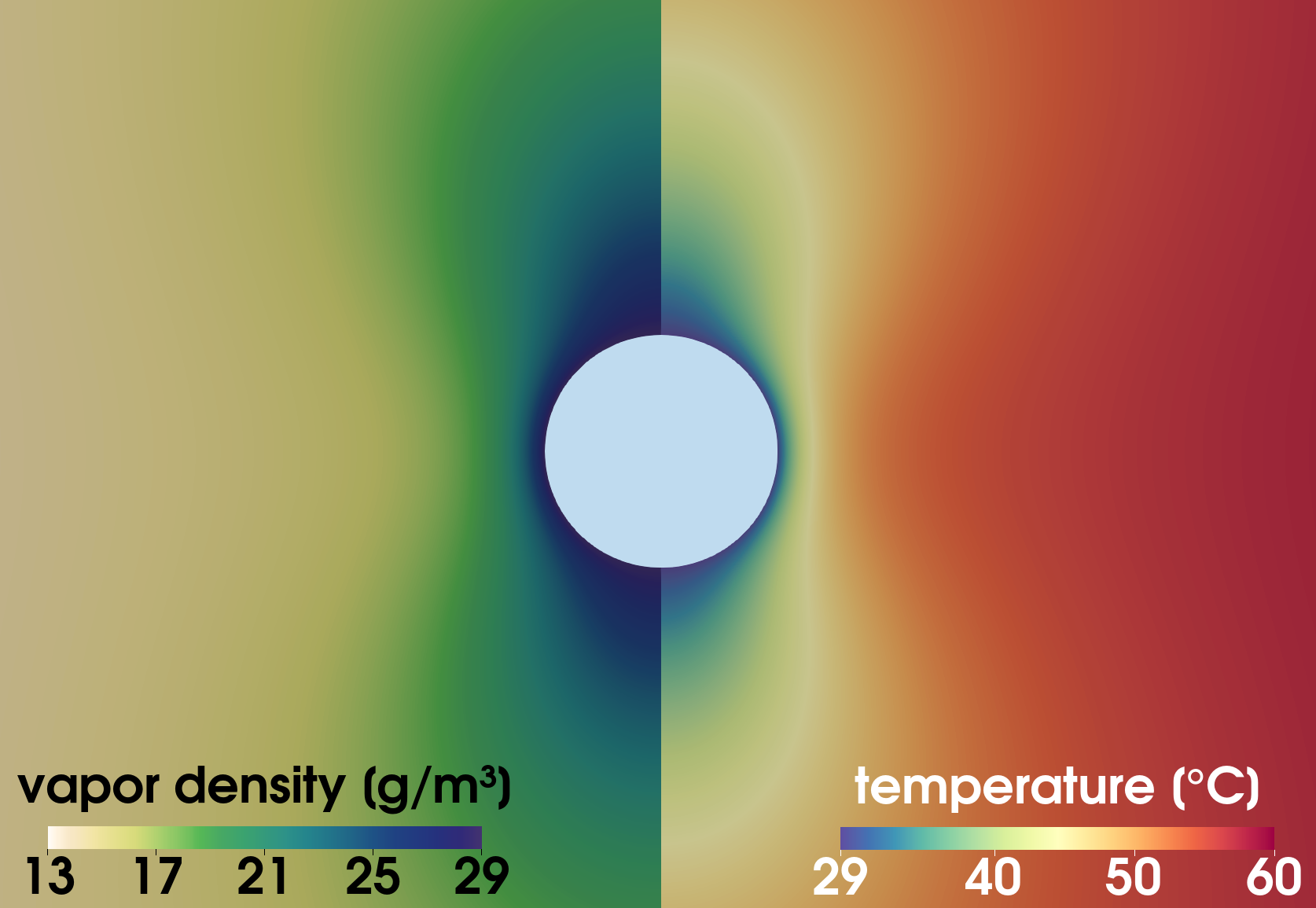}
         \caption{Heat and vapor mass transport at $\text{SPL} = 164\,$dB.}
         \label{fig:heat_and_mass_transfer_acoustic_streaming}
    \end{subfigure}
    \hfill
    \begin{subfigure}[b]{0.48\textwidth}
         \centering
         \includegraphics[height=130pt]{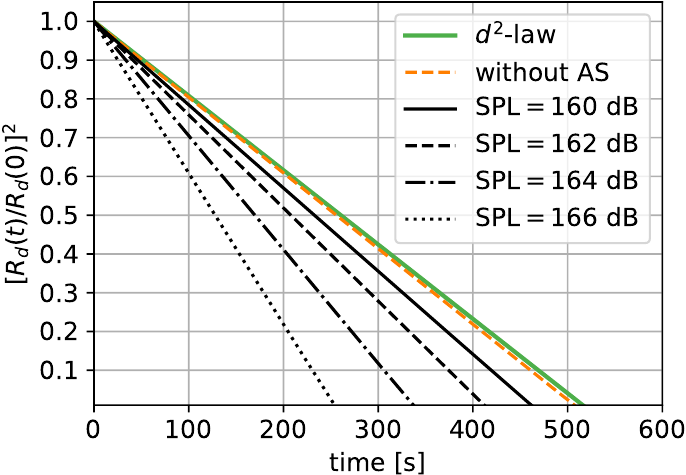}
         \caption{Evolution of the normed squared droplet radius.}
         \label{fig:droplet_radii_acoustic_streaming}
    \end{subfigure}
    \caption{Convective impact of the acoustic streaming (AS) on the evaporation process of a $1\,\mu$l water droplet under the drying conditions $T_\infty = 60\,$\textdegree C and $\text{RH}_\infty = 10\,$\%.}
    \label{fig:evaporation_acoustic_streaming}
\end{figure}
It can be observed that the vapor evaporating from the droplet is blown away from its equator towards its vertical poles. Likewise, the air above and below the droplet has a lower temperature than the one around its equator. As argued in \cite{doss23}, the evaporation rate increases with the normal heat flux from the vertical poles towards the equator of the levitated droplet. Figure \ref{fig:droplet_radii_acoustic_streaming} compares the evolution of the normed squared droplet radius obtained from the $d^2$-law and our direct numerical simulations. The small deviation of our simulated curve without the acoustic streaming from the $d^2$-law was already discussed before. Like the Stokes flow, also the acoustic streaming accelerates the evaporation process of the levitated droplet. Especially at higher sound pressure levels, the lifetime of the droplet is reduced significantly. As a concluding remark, we mention that the droplet levitated at $\text{SPL} = 166\,$dB evaporates approximately as fast as the one exposed to the Stokes flow with $V_\infty = 80\,$cm/s. 

\section{Conclusion} \label{sec_conclusion}

We successfully applied Banach's fixed-point theorem to study the well-posedness of a coupled ODE--PDE system describing the convective evaporation of a spherical droplet. The decoupled problem (for given droplet radius) was shown to admit a unique weak solution using the method of upper and lower solutions. To address the fully coupled ODE--PDE system, we reformulated it as a fixed-point problem. The underlying Volterra operator for the droplet radius was shown to be a contraction for short time intervals. The unique existence of a maximal weak solution finally followed from a topological argument. Apart from these analytical results, we applied the finite element method to perform direct numerical simulations. The $d^2$-law was used to validate our numerical results. We visualized the temperature and vapor mass distributions to investigate the convective impact of two different air flows (Stokes flow and acoustic streaming) around the droplet. Both turned out to accelerate its evaporation process.\\
Recall that our single droplet evaporation model is limited to several simplifications. Instead of assuming the droplet to be isothermal, it would be physically more accurate to compute its temperature from an additional heat equation. Also the shape of the droplet could be derived from a balance of forces including surface tension to account for possible deformations. Finally, the air flow in the ambient gas phase should be modeled by the Navier--Stokes equations. Extending our model in the aforementioned ways leads to new mathematical challenges which are to be addressed in the future.

\section*{Appendix}

In the following, we derive the upper bound \eqref{defect_estimate} for the sum $d = 2 (d_T + d_\rho)$ of the defects \eqref{defect_T} and \eqref{defect_rho} containing all products with $\delta R$, $\delta \dot{R}$, or $\delta v$ that appear in the proof of Proposition \ref{H1_stability}. 
\begin{lemma*}
    The total defect $d$ defined in the proof of Proposition \ref{H1_stability} is bounded by
    \begin{equation} \label{total_defect_bound}
        |d| \leq \varepsilon \big( \| \nabla \delta T \|_{2, \Omega}^2 + \| \nabla \delta \rho \|_{2, \Omega}^2 \big) + C \bigg[ \big( 1 + \| \nabla T_2 \|_{2, \Omega}^2 + \| \nabla \rho_2 \|_{2, \Omega}^2 \big) \delta R^2 + \frac{1}{\varepsilon} \big( \delta \dot{R}^2 + \| \delta T \|_{2, \Omega}^2 + \| \delta \rho \|_{2, \Omega}^2 \big) \bigg] \tag{$\ast$}
    \end{equation}
    with $\varepsilon > 0$ being arbitrarily small.
\end{lemma*}

\begin{proof}
    Recall $|\delta T| \leq T_\infty - T_*$, $|\delta \rho| \leq \rho_* - \rho_\infty$, and $|J_2| \leq J_\infty$ due to the lower and upper bounds \eqref{lower_and_upper_bounds} from Proposition \ref{existence_Rd_given}. Assumption (A6) further implies $\| \delta v \|_{\infty, \Omega} \leq C |\delta R|$ and $\| v_2 \|_{\infty, \Omega} \leq C$. Therefore, the individual terms of the temperature defect \eqref{defect_T} can be handled as follows
    \begin{gather}
        \int_{\Omega} \delta R (\nabla T_2 \cdot v_2) \delta T \, dx \leq C |\delta R| \| \nabla T_2 \|_{2, \Omega} \| \delta T \|_{2, \Omega} \leq C \left( \delta R^2 \| \nabla T_2 \|_{2, \Omega}^2 + \| \delta T \|_{2, \Omega}^2 \right), \notag\\
        - \int_{\Omega} R_2 (\nabla T_2 \cdot \delta v) \delta T \, dx \leq C |\delta R| \| \nabla T_2 \|_{2, \Omega} \| \delta T \|_{2, \Omega} \leq C \left( \delta R^2 \| \nabla T_2 \|_{2, \Omega}^2 + \| \delta T \|_{2, \Omega}^2 \right), \notag\\
        - \int_{\Omega} \delta R \dot{R}_2 (\nabla T_2 \cdot x) \delta T \, dx \leq C |\delta R| \| \nabla T_2 \|_{2, \Omega} \| \delta T \|_{2, \Omega} \leq C \left( \delta R^2 \| \nabla T_2 \|_{2, \Omega}^2 + \| \delta T \|_{2, \Omega}^2 \right), \notag\\
        \int_{\Gamma} \delta R J_2 \delta T \, d\sigma \leq C \left( \delta R^2 + \| \delta T \|_{2, \Gamma}^2 \right) \leq \varepsilon \| \nabla \delta T \|_{2, \Omega}^2 + C \left( \delta R^2 + \frac{1}{\varepsilon} \| \delta T \|_{2, \Omega}^2 \right), \label{boundary_defect_estimate} \tag{$\ast$$\ast$} \\
        \int_{\Omega} \delta R \nabla T_2 \cdot \nabla \delta T \, dx \leq |\delta R| \| \nabla T_2 \|_{2, \Omega} \| \nabla \delta T \|_{2, \Omega} \leq \varepsilon \| \nabla \delta T \|_{2, \Omega}^2 + \frac{C}{\varepsilon} \delta R^2 \| \nabla T_2 \|_{2, \Omega}^2 \notag
    \end{gather}
    using H\"older's and Young's inequalities. Notice that the last inequality in \eqref{boundary_defect_estimate} follows from Lemma \ref{Lemma:1}. The temperature defect term involving $\delta \dot{R}$ is somewhat more delicate and thus treated separately. First of all, we apply the divergence theorem to compute
    \begin{equation*}
        - \int_\Omega R_2 \delta \dot{R} (\nabla T_2 \cdot x) \delta T \, dx = 3 \int_\Omega R_2 \delta \dot{R} T_2 \delta T \, dx + \int_\Gamma R_2 \delta \dot{R} T_2 \delta T \, d\sigma + \int_\Omega R_2 \delta \dot{R} T_2 (\nabla \delta T \cdot x) \, dx
    \end{equation*}
    taking into account $\nabla T_2 \cdot x = \nabla \cdot (T_2 x) - 3 T_2$ in $\Omega$ and $x \cdot n = -1$ on $\Gamma$. The estimate
    \begin{equation*}
        - \int_\Omega R_2 \delta \dot{R} (\nabla T_2 \cdot x) \delta T \, dx \leq \varepsilon \| \nabla \delta T \|_{2, \Omega}^2 + \frac{C}{\varepsilon} \big( \delta \dot{R}^2 + \| \delta T \|_{2, \Omega}^2 \big)
    \end{equation*}
    then follows from Young's inequality and Lemma \ref{Lemma:1}. Altogether, we obtain the upper bound
    \begin{equation*}
        |d_T| \leq \varepsilon \| \nabla \delta T \|_{2, \Omega}^2 + C \bigg[ \big( 1 + \| \nabla T_2 \|_{2, \Omega}^2 \big) \delta R^2 + \frac{1}{\varepsilon} \big( \delta \dot{R}^2 + \| \delta T \|_{2, \Omega}^2 \big) \bigg]
    \end{equation*}
    for the temperature defect. Likewise, we obtain a similar upper bound for the vapor mass defect
    \begin{equation*}
        |d_\rho| \leq \varepsilon \| \nabla \delta \rho \|_{2, \Omega}^2 + C \bigg[ \big( 1 + \| \nabla \rho_2 \|_{2, \Omega}^2 \big) \delta R^2 + \frac{1}{\varepsilon} \big( \delta \dot{R}^2 + \| \delta \rho \|_{2, \Omega}^2 \big) \bigg]
    \end{equation*}
    using the same arguments as before. The upper bound \eqref{total_defect_bound} for the total defect $d = 2 (d_T + d_\rho)$ then follows from the previous estimates for $d_T$ and $d_\rho$ via summation.
\end{proof}

\section*{Funding}

This work was supported by the German Research Foundation (DFG) via Research Training Group 2339 Interfaces, Complex Structures, and Singular Limits.

\section*{Nomenclature}

\begin{longtable}[l]{ll}
    $\Omega$ & gas phase \hfill\\
    $\Gamma$ & droplet surface\\
    $\Gamma_\infty$ & outer boundary of the gas phase\\
    $R$ & droplet radius\\
    $T$ & gas temperature\\
    $T_\infty$ & gas temperature of the drying air\\
    $\rho$ & vapor mass density\\
    $\rho_\infty$ & vapor mass density of the drying air\\
    $\rho_\textit{sat}$ & saturated vapor mass density\\
    $J$ & evaporation rate\\
    $v$ & air flow velocity
\end{longtable}

\def\cprime{$'$} \def\cprime{$'$} \def\cprime{$'$} \def\cprime{$'$}
  \def\cprime{$'$} \def\cprime{$'$} \def\cprime{$'$} \def\cprime{$'$}

\bibliography{literatur}
\bibliographystyle{plain}



\end{document}